\documentclass[12pt]{amsart}
\usepackage{currfile,amsmath}
\usepackage{amsxtra}
\usepackage{amscd}
\usepackage{amsthm}
\usepackage{amsfonts}
\usepackage{amssymb}
\usepackage{eucal}
\usepackage{verbatim}
\usepackage[all]{xy}
\usepackage{mathrsfs}
\usepackage{lineno}
\usepackage{color}
\definecolor{deepjunglegreen}{rgb}{0.0, 0.29, 0.29}
\definecolor{darkspringgreen}{rgb}{0.09, 0.45, 0.27}

\usepackage{stmaryrd}

\makeatletter
\usepackage{etex,etoolbox,needspace}
\pretocmd\section{\Needspace*{4\baselineskip}}{}{}
\makeatletter

\usepackage[pdftex,bookmarks=false,colorlinks=true,citecolor=darkspringgreen,debug=true,
  naturalnames=true,pdfnewwindow=true]{hyperref}

\newtheorem{thm}{Theorem}[subsection]

\newtheorem{cor}[thm]{Corollary}

\newtheorem{lem}[thm]{Lemma}
\newtheorem{prop}[thm]{Proposition}

\theoremstyle{definition}

\theoremstyle{remark}
\newtheorem{rem}[thm]{Remark}

\newcommand{\nc}{\newcommand}
\nc{\renc}{\renewcommand} \nc{\ssec}{\subsection}
\nc{\sssec}{\subsubsection}

\nc{\on}{\operatorname} \nc{\wh}{\widehat}
\nc\ol{\overline} \nc\ul{\underline} \nc\wt{\widetilde}

\newcommand{\red}[1]{{\color{red}#1}}

\nc{\BA}{{\mathbb{A}}} \nc{\BC}{{\mathbb{C}}} \nc{\BF}{{\mathbb{F}}}
\nc{\BD}{{\mathbb{D}}} \nc{\BG}{{\mathbb{G}}} \nc{\BQ}{{\mathbb{Q}}}
\nc{\BM}{{\mathbb{M}}} \nc{\BN}{{\mathbb{N}}} \nc{\BO}{{\mathbb{O}}}
\nc{\BP}{{\mathbb{P}}} \nc{\BR}{{\mathbb{R}}}
\nc{\BZ}{{\mathbb{Z}}} \nc{\BS}{{\mathbb{S}}} \nc{\BW}{{\mathbb{W}}}

\nc{\CA}{{\mathcal{A}}} \nc{\CB}{{\mathcal{B}}} \nc{\CalC}{{\mathcal{C}}} \nc{\CalD}{{\mathcal{D}}}
\nc{\CE}{{\mathcal{E}}} \nc{\CF}{{\mathcal{F}}}
\nc{\CG}{{\mathcal{G}}} \nc{\CH}{{\mathcal{H}}}
\nc{\CI}{{\mathcal{I}}} \nc{\CJ}{{\mathcal{J}}} \nc{\CK}{{\mathcal{K}}} \nc{\CL}{{\mathcal{L}}}
\nc{\CM}{{\mathcal{M}}} \nc{\CN}{{\mathcal{N}}}
\nc{\CO}{{\mathcal{O}}} \nc{\CP}{{\mathcal{P}}}
\nc{\CQ}{{\mathcal{Q}}} \nc{\CR}{{\mathcal{R}}}
\nc{\CS}{{\mathcal{S}}} \nc{\CT}{{\mathcal{T}}}
\nc{\CU}{{\mathcal{U}}} \nc{\CV}{{\mathcal{V}}}  \nc{\CX}{{\mathcal X}} \nc{\CY}{{\mathcal Y}}
\nc{\CW}{{\mathcal{W}}} \nc{\CZ}{{\mathcal{Z}}}

\nc{\scrM}{{\mathscr M}} \nc{\scrP}{{\mathscr P}}

\nc{\cM}{{\check{\mathcal M}}{}} \nc{\csM}{{\check{\mathcal A}}{}}
\nc{\oM}{{\overset{\circ}{\mathcal M}}{}}
\nc{\obM}{{\overset{\circ}{\mathbf M}}{}}
\nc{\oCA}{{\overset{\circ}{\mathcal A}}{}}
\nc{\obA}{{\overset{\circ}{\mathbf A}}{}}
\nc{\ooM}{{\overset{\circ}{M}}{}}
\nc{\osM}{{\overset{\circ}{\mathsf M}}{}}
\nc{\vM}{{\overset{\bullet}{\mathcal M}}{}}
\nc{\nM}{{\underset{\bullet}{\mathcal M}}{}}
\nc{\oD}{{\overset{\circ}{\mathcal D}}{}}
\nc{\obD}{{\overset{\circ}{\mathbf D}}{}}
\nc{\oA}{{\overset{\circ}{\mathbb A}}{}}
\nc{\op}{{\overset{\bullet}{\mathbf p}}{}}
\nc{\cp}{{\overset{\circ}{\mathbf p}}{}}
\nc{\oU}{{\overset{\bullet}{\mathcal U}}{}}
\nc{\ofZ}{{\overset{\circ}{\mathfrak Z}}{}}

\nc{\ff}{{\mathfrak{f}}} \nc{\fv}{{\mathfrak{v}}}
\nc{\fa}{{\mathfrak{a}}} \nc{\fb}{{\mathfrak{b}}}
\nc{\fd}{{\mathfrak{d}}} \nc{\fe}{{\mathfrak{e}}}
\nc{\fg}{{\mathfrak{g}}} \nc{\fgl}{{\mathfrak{gl}}}
\nc{\fh}{{\mathfrak{h}}} \nc{\fri}{{\mathfrak{i}}}
\nc{\fj}{{\mathfrak{j}}} \nc{\fk}{{\mathfrak{k}}} \nc{\fl}{{\mathfrak{l}}}
\nc{\fm}{{\mathfrak{m}}} \nc{\fn}{{\mathfrak{n}}}
\nc{\ft}{{\mathfrak{t}}} \nc{\fu}{{\mathfrak{u}}}
\nc{\fw}{{\mathfrak{w}}} \nc{\fz}{{\mathfrak{z}}}
\nc{\fp}{{\mathfrak{p}}} \nc{\fq}{{\mathfrak{q}}} \nc{\frr}{{\mathfrak{r}}}
\nc{\fs}{{\mathfrak{s}}} \nc{\fsl}{{\mathfrak{sl}}}
\nc{\fso}{{\mathfrak{so}}} \nc{\fsp}{{\mathfrak{sp}}} \nc{\osp}{{\mathfrak{osp}}}
\nc{\hsl}{{\widehat{\mathfrak{sl}}}}
\nc{\hgl}{{\widehat{\mathfrak{gl}}}}
\nc{\hg}{{\widehat{\mathfrak{g}}}}
\nc{\chg}{{\widehat{\mathfrak{g}}}{}^\vee}
\nc{\hn}{{\widehat{\mathfrak{n}}}}
\nc{\chn}{{\widehat{\mathfrak{n}}}{}^\vee}

\nc{\fA}{{\mathfrak{A}}} \nc{\fB}{{\mathfrak{B}}} \nc{\fC}{{\mathfrak{C}}}
\nc{\fD}{{\mathfrak{D}}} \nc{\fE}{{\mathfrak{E}}}
\nc{\fF}{{\mathfrak{F}}} \nc{\fG}{{\mathfrak{G}}} \nc{\fH}{{\mathfrak{H}}}
\nc{\fI}{{\mathfrak{I}}} \nc{\fJ}{{\mathfrak{J}}}
\nc{\fK}{{\mathfrak{K}}} \nc{\fL}{{\mathfrak{L}}}
\nc{\fM}{{\mathfrak{M}}} \nc{\fN}{{\mathfrak{N}}}
\nc{\frP}{{\mathfrak{P}}} \nc{\fQ}{{\mathfrak{Q}}} \nc{\fR}{{\mathfrak{R}}}
\nc{\fS}{{\mathfrak{S}}} \nc{\fT}{{\mathfrak{T}}} \nc{\fU}{{\mathfrak{U}}}
\nc{\fV}{{\mathfrak{V}}} \nc{\fW}{{\mathfrak{W}}}
\nc{\fX}{{\mathfrak{X}}} \nc{\fY}{{\mathfrak{Y}}}
\nc{\fZ}{{\mathfrak{Z}}}

\nc{\ba}{{\mathbf{a}}}
\nc{\bb}{{\mathbf{b}}} \nc{\bc}{{\mathbf{c}}} \nc{\be}{{\mathbf{e}}}
\nc{\bg}{{\mathbf{g}}} \nc{\bj}{{\mathbf{j}}} \nc{\bm}{{\mathbf{m}}}
\nc{\bn}{{\mathbf{n}}} \nc{\bp}{{\mathbf{p}}}
\nc{\bq}{{\mathbf{q}}} \nc{\br}{{\mathbf{r}}} \nc{\bs}{{\mathbf{s}}}
\nc{\bt}{{\mathbf{t}}} \nc{\bfu}{{\mathbf{u}}} \nc{\bv}{{\mathbf{v}}}
\nc{\bx}{{\mathbf{x}}} \nc{\by}{{\mathbf{y}}} \nc{\bz}{{\mathbf{z}}}
\nc{\bw}{{\mathbf{w}}} \nc{\bA}{{\mathbf{A}}}
\nc{\bB}{{\mathbf{B}}} \nc{\bC}{{\mathbf{C}}}
\nc{\bD}{{\mathbf{D}}} \nc{\bF}{{\mathbf{F}}} \nc{\bG}{{\mathbf{G}}}
\nc{\bH}{{\mathbf{H}}} \nc{\bI}{{\mathbf{I}}} \nc{\bJ}{{\mathbf{J}}}
\nc{\bK}{{\mathbf{K}}} \nc{\bL}{{\mathbf{L}}} \nc{\bM}{{\mathbf{M}}} \nc{\bN}{{\mathbf{N}}}
\nc{\bO}{{\mathbf{O}}} \nc{\bS}{{\mathbf{S}}} \nc{\bT}{{\mathbf{T}}}
\nc{\bU}{{\mathbf{U}}} \nc{\bV}{{\mathbf{V}}} \nc{\bW}{{\mathbf{W}}}
\nc{\bX}{{\mathbf{X}}}
\nc{\bY}{{\mathbf{Y}}} \nc{\bP}{{\mathbf{P}}}
\nc{\bZ}{{\mathbf{Z}}} \nc{\bh}{{\mathbf{h}}}

\nc{\sA}{{\mathsf{A}}} \nc{\sB}{{\mathsf{B}}}
\nc{\sC}{{\mathsf{C}}} \nc{\sD}{{\mathsf{D}}}
\nc{\sE}{{\mathsf{E}}} \nc{\sF}{{\mathsf{F}}} \nc{\sG}{{\mathsf{G}}}
\nc{\sI}{{\mathsf{I}}} \nc{\sK}{{\mathsf{K}}} \nc{\sL}{{\mathsf{L}}}
\nc{\sfm}{{\mathsf{m}}} \nc{\sM}{{\mathsf{M}}} \nc{\sN}{{\mathsf{N}}}
\nc{\sO}{{\mathsf{O}}} \nc{\sQ}{{\mathsf{Q}}} \nc{\sP}{{\mathsf{P}}} \nc{\sR}{{\mathsf{R}}}
\nc{\sT}{{\mathsf{T}}} \nc{\sZ}{{\mathsf{Z}}}
\nc{\sV}{{\mathsf{V}}} \nc{\sW}{{\mathsf{W}}}
\nc{\sfp}{{\mathsf{p}}} \nc{\sq}{{\mathsf{q}}} \nc{\sr}{{\mathsf{r}}}
\nc{\sfs}{{\mathsf{s}}} \nc{\st}{{\mathsf{t}}} \nc{\sfb}{{\mathsf{b}}}
\nc{\sfc}{{\mathsf{c}}} \nc{\sd}{{\mathsf{d}}}
\nc{\sz}{{\mathsf{z}}}

\nc{\tA}{{\widetilde{\mathbf{A}}}}
\nc{\tB}{{\widetilde{\mathcal{B}}}}
\nc{\tg}{{\widetilde{\mathfrak{g}}}} \nc{\tG}{{\widetilde{G}}}
\nc{\TM}{{\widetilde{\mathbb{M}}}{}}
\nc{\tO}{{\widetilde{\mathsf{O}}}{}}
\nc{\tU}{{\widetilde{\mathfrak{U}}}{}} \nc{\TZ}{{\tilde{Z}}}
\nc{\tx}{{\tilde{x}}} \nc{\tbv}{{\tilde{\bv}}}
\nc{\tfP}{{\widetilde{\mathfrak{P}}}{}} \nc{\tz}{{\tilde{\zeta}}}
\nc{\tmu}{{\tilde{\mu}}}

\nc{\urho}{\underline{\rho}} \nc{\uB}{\underline{B}}
\nc{\uC}{{\underline{\mathbb{C}}}} \nc{\ui}{\underline{i}}
\nc{\uj}{\underline{j}} \nc{\ofP}{{\overline{\mathfrak{P}}}}
\nc{\oB}{{\overline{\mathcal{B}}}}
\nc{\og}{{\overline{\mathfrak{g}}}} \nc{\oI}{{\overline{I}}}

\nc{\eps}{\varepsilon} \nc{\hrho}{{\hat{\rho}}} \nc{\bsigma}{{\boldsymbol{\sigma}}}
\nc{\blambda}{{\boldsymbol{\lambda}}} \nc{\bmu}{{\boldsymbol{\mu}}} \nc{\bnu}{{\boldsymbol{\nu}}}
\nc{\btheta}{{\boldsymbol{\theta}}} \nc{\bzeta}{{\boldsymbol{\zeta}}} \nc{\bta}{{\boldsymbol{\eta}}}
\nc{\bomega}{{\boldsymbol{\omega}}} \nc{\bxi}{{\boldsymbol{\xi}}} \nc{\brho}{{\boldsymbol{\rho}}}

\nc{\one}{{\mathbf{1}}} \nc{\two}{{\mathbf{t}}}

\nc{\Sym}{\mathop{\operatorname{\rm Sym}}} \nc{\RH}{{\mathop{\rm RH}}} \nc{\RT}{{\mathop{\rm RT}}}
\nc{\Tot}{{\mathop{\operatorname{\rm Tot}}}}
\nc{\Spec}{\mathop{\operatorname{\rm Spec}}}
\nc{\Ker}{{\mathop{\operatorname{\rm Ker}}}}
\nc{\Isom}{{\mathop{\operatorname{\rm Isom}}}}
\nc{\Hilb}{{\mathop{\operatorname{\rm Hilb}}}}
\nc{\deeq}{{\mathop{\operatorname{\rm deeq}}}}
\nc{\End}{{\mathop{\operatorname{\rm End}}}}
\nc{\Ext}{{\mathop{\operatorname{\rm Ext}}}}
\nc{\Hom}{{\mathop{\operatorname{\rm Hom}}}}
\nc{\CHom}{{\mathop{\operatorname{{\mathcal{H}}\!\it om}}}}
\nc{\GL}{{\mathop{\operatorname{\rm GL}}}}
\nc{\PGL}{{\mathop{\operatorname{\rm PGL}}}}
\nc{\SL}{{\mathop{\operatorname{\rm SL}}}}
\nc{\SO}{{\mathop{\operatorname{\rm SO}}}}
\nc{\Sp}{{\mathop{\operatorname{\rm Sp}}}}
\nc{\Mp}{{\widetilde{\mathop{\operatorname{\rm Sp}}}}}
\nc{\OSp}{{\mathop{\operatorname{\rm SOSp}}}}
\nc{\gr}{{\mathop{\operatorname{\rm gr}}}}
\nc{\Id}{{\mathop{\operatorname{\rm Id}}}}
\nc{\perf}{{\mathop{\operatorname{\rm perf}}}}
\nc{\defi}{{\mathop{\operatorname{\rm def}}}}
\nc{\length}{{\mathop{\operatorname{\rm length}}}}
\nc{\supp}{{\mathop{\operatorname{\rm supp}}}}
\nc{\HC}{{\mathcal H}{\mathcal C}}

\nc{\pr}{{\operatorname{pr}}}
\nc{\Cliff}{{\mathsf{Cliff}}}
\nc{\Loc}{{\operatorname{Loc}}} \nc{\LOC}{{\bf{Loc}}}
\nc{\Fl}{{\mathbf{Fl}}} \nc{\Ffl}{{\mathcal{F}\ell}}
\nc{\Fib}{{\mathsf{Fib}}}
\nc{\Coh}{{\operatorname{Coh}}} \nc{\FCoh}{{\mathsf{FCoh}}}
\nc{\Perf}{{\mathsf{Perf}}}
\nc{\wtimes}{\mathbin{\widetilde\times}}

\nc{\reg}{{\text{\rm reg}}}
\nc{\self}{{\text{\rm self}}}
\nc{\gvee}{{\mathfrak g}^{\!\scriptscriptstyle\vee}}
\nc{\tvee}{{\mathfrak t}^{\!\scriptscriptstyle\vee}}
\nc{\nvee}{{\mathfrak n}^{\!\scriptscriptstyle\vee}}
\nc{\bvee}{{\mathfrak b}^{\!\scriptscriptstyle\vee}}
\nc{\rhovee}{\rho^{\!\scriptscriptstyle\vee}}

\newcommand{\svee}{{\!\scriptscriptstyle\vee}}
       
\nc{\cplus}{{\mathbf{C}_+}} \nc{\cminus}{{\mathbf{C}_-}}
\nc{\cthree}{{\mathbf{C}_*}} \nc{\Qbar}{{\bar{Q}}}

\newcommand{\bCD}{\vphantom{j^{X^2}}\smash{\overset{\bullet}{\vphantom{\rule{0pt}{0.55em}}\smash{\mathcal D}}}}

\newcommand\iso{\mathbin{\vphantom{j^{X^2}}\smash{\overset{\sim}{\vphantom{\rule{0pt}{0.20em}}\smash{\longrightarrow}}}}}
\nc{\Gtimes}{\vphantom{j^{X^2}}\smash{\overset{G}{\vphantom{\rule{0pt}{0.30em}}\smash{\times}}}}
\nc{\sGtimes}{\vphantom{j^{X^2}}\smash{\overset{\mathsf G}{\vphantom{\rule{0pt}{0.30em}}\smash{\times}}}}

\nc{\bOmega}{{\overline{\Omega}}}

\nc{\seq}[1]{\stackrel{#1}{\sim}}

\nc{\aff}{{\operatorname{aff}}}
\nc{\fin}{{\operatorname{fin}}}
\nc{\mir}{{\operatorname{mir}}}
\nc{\triv}{{\operatorname{triv}}}
\nc{\ext}{{\operatorname{ext}}}
\nc{\righ}{{\operatorname{right}}}
\nc{\lef}{{\operatorname{left}}}
\nc{\forg}{{\operatorname{forg}}}
\nc{\fid}{{\operatorname{fd}}}
\nc{\odd}{{\operatorname{odd}}}
\nc{\even}{{\operatorname{even}}}
\nc{\modu}{{\operatorname{-mod}}}
\nc{\Dmod}{{\operatorname{D-mod}_{-1/2}^{G_{\CO}}}}
\nc{\Dmo}{\operatorname{D-mod}_{-1/2}}
\nc{\Gr}{{\operatorname{Gr}}} \nc{\LGr}{{\operatorname{LGr}}}
\nc{\tGr}{{\widetilde{\operatorname{Gr}}}} \nc{\tLGr}{{\widetilde{\operatorname{LGr}}}}
\nc{\GR}{{\mathbf{Gr}}}
\nc{\tGR}{{\widetilde{\mathbf{Gr}}}}
\nc{\FT}{{\operatorname{FT}}}
\nc{\Mat}{{\operatorname{Mat}}}
\nc{\MSt}{{\operatorname{MSt}}}
\nc{\sph}{{\operatorname{sph}}}
\nc{\Perv}{{\operatorname{Perv}}}
\nc{\Rep}{{\operatorname{Rep}}}
\nc{\Ind}{{\operatorname{Ind}}}
\nc{\IC}{{\operatorname{IC}}}
\nc{\Bun}{{\operatorname{Bun}}}
\nc{\Proj}{{\operatorname{Proj}}}
\nc{\Stab}{{\operatorname{Stab}}}
\nc{\Sq}{{\operatorname{Sq}}}
\nc{\pt}{{\operatorname{pt}}}
\nc{\sico}{{\operatorname{sc}}}
\nc{\bfmu}{{\boldsymbol{\mu}}}
\nc{\bfomega}{{\boldsymbol{\omega}}}
\nc{\bGamma}{{\boldsymbol{\Gamma}}}
\nc{\calM}{\mathcal M}
\nc{\calA}{\mathcal A}
\nc{\calO}{\mathcal O}
\nc{\CC}{\mathbb C}
\nc{\calN}{\mathcal N}
\nc{\grg}{\mathfrak g}
\nc{\dslash}{/\!\!/}
\nc{\tslash}{/\!\!/\!\!/}
\nc\grt{\mathfrak t}
\nc\bfM{\mathbf M}
\nc\bfN{\mathbf N}
\nc\Sig{\Sigma}
\nc\ZZ{\mathbb{Z}}
\nc\calC{\mathcal C}
\nc\calF{\mathcal F}
\nc\calX{\mathcal X}
\nc\calY{\mathcal Y}
\nc\QCoh{\operatorname{QCoh}}
\nc\IndCoh{\operatorname{IndCoh}}
\nc\Maps{\operatorname{Maps}}
\newcommand\Hecke{\operatorname{Hecke}}
\nc{\calD}{\mathcal D}
\nc\bfO{\mathbf O}

\nc\GG{\mathbb G}
\nc\calK{\mathcal K}
\nc{\calG}{\mathcal G}
\nc\RHom{\operatorname{RHom}}
\nc\Res{\operatorname{Res}}
\nc\Av{\operatorname{Av}}
\nc\grs{\mathfrak s}
\nc{\tilX}{\widetilde X}
\nc\calB{\mathcal B}
\nc\calS{\mathcal S}
\nc\calT{\mathcal T}
\nc\calZ{\mathcal Z}
\nc\LS{\operatorname{LocSys}}
\nc\bfL{\on{\mathbf L}}
\nc{\gru}{\mathfrak u}

\newcommand*\circled[1]
{*{#1}}
\makeatletter
\newcommand{\raisemath}[1]{\mathpalette{\raisem@th{#1}}}
\newcommand{\raisem@th}[3]{\raisebox{#1}{$#2#3$}}
\nc{\binlim}[2][]{\def\@tempa{#1}\@ifnextchar^{\@binlim{#2}}{\@binlim{#2}^{}}}
\def\@binlim#1^#2{\mathbin{\@ifempty{#2}{\mathop{#1}}{\mathop{#1}\@xp\displaylimits\@tempa^{#2}}}}

\makeatother

\nc\cX{{\mathcal X}}

\newcommand{\dbkts}[1]{[\![#1]\!]}
\newcommand{\dprts}[1]{(\!(#1)\!)}
\nc\Gm{{\mathbb G_m}}
\nc{\isoto}
    {\stackrel{\displaystyle\raisemath{-.2ex}{\sim}}\to}
\nc\cD\CalD
\nc\setm\setminus
\nc\cE\CE
\renc\Hecke{\mathit{\CH\kern-.2ex ecke}}
\nc\bsl\backslash
\nc\Fq{\mathbb F_q}
\nc\x\times
\nc\bGO{{\bG_\bO}}
\nc\bla\blambda
\nc\blt\bullet
\nc\opp{{\on{op}}}
\nc\srel\stackrel
\nc\oast\circledast
\nc\tbx{\binlim{\widetilde\boxtimes{}}}
\nc\into\hookrightarrow
\nc\otto\leftrightarrow
\renc\setminus\smallsetminus
\nc\phitau{\varphi\tau}
\newenvironment{i-ii-iii}{%
\begin{enumerate}
}%
{\end{enumerate}}
\nc\ceil[1]{\lceil#1\rceil}  \nc\floor[1]{\lfloor#1\rfloor}
\nc\Lie{\on{Lie}}
\nc\sS{{\mathsf S}}
\nc\vvv{\ensuremath{\red\surd}}
\nc\la\lambda
\def\arxiv#1{\href{http://arxiv.org/abs/#1}{\tt arXiv:#1}} 
\nc\kap{\kappa}
\nc\gra{\mathfrak a}
\nc\gl{\mathfrak{gl}}
\nc\sTr{\operatorname{sTr}}
\nc\hatG{\widehat{G}}
\nc\calL{\mathcal L}
\nc\Whit{\operatorname{Whit}}
\nc\KL{\operatorname{KL}}

\newcommand\PP{\mathbb P}

\newcommand\B{{\mathrm B}}

\makeatletter
\renewcommand{\subsection}{\@startsection{subsection}{2}{0pt}{-3ex
plus -1ex minus -0.2ex}{-2mm plus -0pt minus
-2pt}{\normalfont\bfseries}} \makeatother

\makeatletter\let\@wraptoccontribs\wraptoccontribs
\makeatother

\numberwithin{equation}{subsection}
\allowdisplaybreaks
\nc\mto{\mapsto }
\nc\en{\enspace }
\nc\ome{\omega}

\newcommand{\FL}{\mathbf{Fl}}
\newcommand{\DMOD}{\on{D-mod}_\kappa}
\newcommand{\DMODD}{\on{D-mod}_{-\kappa + 2\kappa_c}}
\newcommand{\wpi}{\widetilde{\pi}}
\newcommand{\afl}{\on{Fl}_G}
\newcommand{\tfl}{\widetilde{\on{Fl}}_G}
\newcommand{\mon}{\on{-mon}}
\newcommand{\gk}{\widehat{\fg}_\kappa}
\newcommand{\gkk}{\widehat{\fg}_{-\kappa + 2\kappa_c}}
\renewcommand{\mod}{\on{-mod}}
\nc{\tili}{\widetilde i}

\begin{document}

\author[A.Braverman]{Alexander Braverman}
\address{Department of Mathematics, University of Toronto and Perimeter Institute
of Theoretical Physics, Waterloo, Ontario, Canada, N2L 2Y5}
\email{braval@math.toronto.edu}

\author[G.Dhillon]{Gurbir Dhillon}
\address{UCLA Mathematics Department, Los Angeles ,CA 90095-1555, USA}
\email{gsd@math.ucla.edu}

\author[M.Finkelberg]{Michael Finkelberg}
\address{Einstein Institute of Mathematics, The Hebrew University of Jerusalem,
  Edmond J. Safra Campus, Giv’at Ram, Jerusalem, 91904, Israel;
\newline  National Research University Higher School of Economics;
\newline Skolkovo Institute of Science and Technology}
\email{fnklberg@gmail.com}

\author[S.Raskin]{Sam Raskin}
\address{Yale University, Department of Mathematics, 219 Prospect st, New Haven, CT 06511, USA}
\email{sam.raskin@yale.edu}

\author[R.Travkin]{Roman Travkin}
\address{Skolkovo Institute of Science and Technology, Moscow, Russia}
\email{roman.travkin2012@gmail.com}

\contrib[With appendices by]{Gurbir Dhillon and Theo Johnson-Freyd}

\title
    {Coulomb branches of noncotangent type}
    \maketitle

    \begin{abstract}
We propose a construction of the Coulomb branch of a $3d\ \CN=4$ gauge theory corresponding to a
choice of a connected reductive group $G$ and a symplectic finite-dimensional reprsentation $\bM$
of $G$, satisfying certain anomaly cancellation condition. This extends the construction of~\cite{bfna}
(where it was assumed that $\bfM=\bfN\oplus \bfN^*$ for some representation $\bfN$ of $G$).
Our construction goes through certain ``universal'' ring object in the twisted derived Satake
category of the symplectic group $\Sp(2n)$. The construction of this object uses a categorical version
of the Weil representation; we also compute the image of this object under the (twisted) derived
Satake equivalence and show that it can be obtained from the theta-sheaf~\cite{l,ll} on
$\Bun_{\Sp(2n)}(\BP^1)$ via certain Radon transform. We also discuss applications of our construction
to a potential mathematical construction of $S$-duality for super-symmetric boundary conditions in
4-dimensional gauge theory and to (some extension of) the conjectures of Ben-Zvi, Sakellaridis and
Venkatesh.
    \end{abstract}

\tableofcontents

\section{Introduction}
\subsection{Symplectic duality}
Let $X$ be an algebraic variety over $\CC$. We say that $X$ is singular symplectic (or $X$ has symplectic singularities) if

(1) $X$ is a normal Poisson variety;

(2) There exists a smooth dense open subset $U$ of $X$ on which the Poisson structure comes from a symplectic structure. We shall denote by $\omega$ the corresponding symplectic form.

(3) There exists a resolution of singularities $\pi\colon \tilX\to X$ such that $\pi^*\omega$ has no poles on $\tilX$.

We say that $X$ is a conical symplectic singularity if in addition to~(1)-(3) above one has a
$\CC^{\times}$-action on $X$ which acts on $\ome$ with some positive weight and which contracts all of $X$ to one point.

A symplectic resolution  $\pi\colon \tilX\to X$  is a proper and birational morphism $\pi$ such that
$\pi^*\omega$ extends to a symplectic form on $\tilX$.
Here is one example. Let $\grg$ be a semi-simple Lie algebra over $\CC$ and let $\calN_{\grg}\subset \grg^*$ be its nilpotent cone. Let $\calB$ denote the flag variety of $\grg$. Then the Springer map $\pi\colon T^*\calB\to \calN_{\grg}$ is proper and birational, so if we let $X=\calN_{\grg}, \tilX=T^*\calB$ we get a symplectic resolution.

The idea of symplectic duality is this: often conical symplectic singularities come in ``dual" pairs $(X,X^*)$ (the assignment $X\to X^*$ is by no means a functor; we just have a lot of interesting examples of dual pairs). What does it mean that $X$ and $X^*$ are dual? This is in general not easy to tell, but many geometric questions about $X$ should be equivalent to some {\em other} geometric questions about $X^*$.
For example,  we should have
$\dim H^\bullet(\tilX,\CC)=\dim H^\bullet(\tilX^*,\CC)$
(but these spaces are not supposed to be canonically isomorphic).
We refer the reader to \cite{BPW}, \cite{BLPW} for more details. There should be a lot of other connections between $X$ and $X^*$ which will take much longer to describe; we refer the reader to
{\em loc.cit.}\ for the description of these properties as well as for examples.

\subsection{3-dimensional $\CN=4$ quantum field theories}\label{3dN4}
One source of dual pairs $(X,X^*)$ comes from quantum field theory in the following way.
Physicists have a notion of 3-dimensional $\CN=4$ super-symmetric quantum field theory. Any such theory $\calT$ is supposed to have a well-defined moduli space of vacua $\calM(\calT)$. This space is complicated, but it should have two special pieces called the Higgs and the Coulomb branch; we shall denote these by $\calM_H(\calT)$ and $\calM_C(\calT)$. They are supposed to be (singular) symplectic complex algebraic varieties (in fact, they don't even have to be algebraic but for simplicity we shall only consider examples when they are).

Let $G$ be a complex reductive algebraic group and let $\bfM$ be a symplectic vector space with a Hamiltonian action of $G$. Then to the pair $(G,\bfM)$ one is supposed to associate
a theory $\calT(G,\bfM)$ provided that $\bfM$ satisfies certain {\em anomaly cancellation condition}, which can be formulated as follows. The representation $\bfM$ defines a homomorphism $G\to \Sp(\bfM)$ and thus a homomorphism
$\pi_4(G)\to \pi_4(\Sp(\bfM))=\ZZ/2\ZZ$. The anomaly cancellation condition is the condition that
this homomorphism is trivial.  Without going to further details at the moment we would like to
emphasize the following:

1) Any $\bfM$ of the form $T^*\bfN=\bfN\oplus\bfN^*$ where $\bfN$ is some representation $G$ satisfies this condition.

2) The anomaly cancellation condition is a ``$\ZZ/2\ZZ$-condition" (later on we are going to formulate it more algebraically).

\medskip
\noindent
Assume that we are given $\bfM$ as above for which the anomaly cancellation condition is satisfied.  Then the theory $\calT(G,\bfM)$ is called {\em gauge theory with gauge group $G$ and matter $\bfM$}. Its Higgs branch is expected to be equal to $\bfM\tslash G$: the Hamiltonian reduction of $\bfM$ with respect to $G$.
In particular, all Nakajima quiver varieties arise in this way (the corresponding theories are called {\em quiver gauge theories}).

The corresponding Coulomb branches are much trickier to define. Physicists had some expectations about those but no rigorous definition in general (only some examples).
The idea is that at least in the conical case the pair $(\calM_H(\calT),\calM_C(\calT))$ should produce an example of a dual symplectic pair.
A mathematical approach to the definition of Coulomb branches was proposed in~\cite{n}.
A rigorous definition of the Coulomb branches $\CM_C(G,\bM)$ is given in~\cite{bfna} under the
assumption that $\bfM=T^*\bfN=\bfN\oplus\bfN^*$ for some representation $\bfN$ of
$G$.\footnote{In addition to the Coulomb branch $\CM_C(G,\bM)$, in~\cite[Remark 3.14]{bfna}
the authors define the $K$-theoretic Coulomb branch $\CM_C^K(G,\bM)$
under the same assumption (physically, it should correspond to the Coulomb branch of the corresponding 4d gauge theory of ${\mathbb R}^3\times S^1$). We would like to emphasize that at this point we are {\em not} able to extend this construction to arbitrary symplectic
$\bM$ with anomaly cancellation condition. Such an extension is given in~\cite{t}.}
The varieties $\calM_C(G,\bfM)$ are normal, affine, Poisson, generically symplectic and satisfy the
monopole formula. We expect that they are singular symplectic, but we can not prove this in general,
cf.~\cite{we}.
The main ingredient in the definition is the geometry of the affine Grassmannian $\Gr_G$ of $G$.
In~\cite{bfna,bfnb,bfnc} these varieties are computed in many cases (in particular, in the case of so called quiver gauge theories --- it turns out that one can associate a pair $(G,\bfN)$ to any framed quiver).
The quantizations of these varieties are also studied, as well as their (Poisson) deformations and (partial) resolutions.

\subsection{Coulomb branches via ring objects in the derived Satake category}
Let $\CK=\BC\dprts t\supset\CO=\BC\dbkts t$. The affine Grassmannian ind-scheme $\Gr_G=G_\CK/G_\CO$
is the moduli space of $G$-bundles on the formal disc equipped with a trivialization on the
punctured formal disc. One can consider the {\em derived Satake category}
$D_{G_\CO}(\Gr_G)$.\footnote{In fact we are going to work with a certain renormalized version of it, cf.~\S\ref{aff gr}.}
This is a factorization monoidal category which is monoidally equivalent to
$D^{G^{\vee}}(\Sym^\bullet(\gvee[-2]))$: the derived category of dg-modules over
$\Sym^\bullet(\gvee[-2])$ endowed with a compatible action of $G^{\vee}$
(the monoidal structure on this category is just given by tensor product over
$\Sym^\bullet(\gvee[-2])$); we shall denote the corresponding functor from $D_{G_\CO}(\Gr_G)$ to
$D^{G^{\vee}}(\Sym^\bullet(\gvee[-2])$ by $\Phi_G$.
In \cite{bfnc} we have attached to any $\bfN$ as above a certain ring object $\calA_{G,\bfM}$ in
$D_{G_\CO}(\Gr_G)$ (here as before we set $\bfM=T^*\bfN$) such
that the algebra of functions on $\calM_C(G,\bfM)$ is equal to $H^\bullet_{G_\CO}(\Gr_G,\calA_{G,\bfM})$
(this cohomology has an algebra structure coming from the fact that $\calA_{G,\bM}$ is a ring object).

\subsection{Ring objects for general $\bfM$ and twisted Satake category}
One of the main goals of this paper is to construct the ring object $\calA_{G,\bfM}$ for
arbitrary symplectic representation $\bfM$ satisfying the anomaly cancellation
condition.\footnote{Another construction of the Coulomb branch of a $3d\ \CN=4$ gauge theory
  in the noncotangent case was proposed by C.~Teleman~\cite{t}.}
In fact, we can construct the ring object $\calA_{G,\bfM}$ for any symplectic $\bfM$ but instead of
being an object of the derived Satake category $D_{G_\CO}(\Gr_G)$ it will be an object of a certain
twisted version of it. More precisely, the representation $\bfM$ defines certain determinant line
bundle $\CL_{\bfM}$ on $\Gr_G$ which is equipped with certain multiplicative and factorization
structure; we shall denote by $\CL_{\bfM}^0$ the total space of this bundle without the zero section.
The line bundle $\CL_{\bfM}$ is also $G_\CO$-equivariant. In particular, for any $\tau\in \BC$ one can
consider the category $D_\tau^{G_\CO}(\Gr_G)$ of $G_\CO$-equivariant sheaves on $\CL_{\bfM}^0$ that are
$\BC^\times$-monodromic with monodromy $q=e^{2\pi i \tau}$. This category is again factorization monoidal
(because of the above multiplicative and factorization structure on $\CL_\bfM$). If $\tau$ is a
rational number and $\CL_\bfM^\tau$ exists as a multiplicative factorization line bundle on $\Gr_G$,
the twisted category $D_\tau^{G_\CO}(\Gr_G)$ is naturally equivalent to $D_{G_\CO}(\Gr_G)$
(as a factorization monoidal category).

In this paper we shall construct a ring object $\calA_{G,\bfM}\in D_{-1/2}^{G_\CO}(\Gr_G)$. It turns out
(see~Proposition~\ref{theo}) that the anomaly cancellation condition is equivalent to the existence of
a factorization multiplicative square root of $\calD_\bfM$ as a {\em super} line bundle.
So, we can construct the ring object $\calA_{G,\bfM}$ but it will be untwisted only if the anomaly
cancellation condition is satisfied. In particular, we can take its $G_\CO$-equivariant cohomology
(and thus define the algebra of functions on the corresponding Coulomb branch) only under the anomaly
cancellation assumption.


\subsection{The universal twisted ring object}
In fact in order to construct the ring object $\calA_{G,\bfM}$ for any $G$ and $\bfM$ it is enough to
do it when $G=\Sp(2n)$ and $\bfM=\CC^{2n}$ is its tautological representation. The reason is as
follows. Assume first that $\bM=T^*\bN$ and let $i\colon G'\to G$ be a homomorphism of connected
reductive groups. It induces a morphism $\tili\colon \Gr_{G'}\to \Gr_G$, and it follows from the construction of \cite{bfnc} that $\calA_{G',\bfM}=\tili^!\calA_{G,\bfM}$. Assuming that the same is true for arbitrary $\bfM$ and since the symplectic representation $\bfM$ is the same as a homomorphism $G\to \Sp(\bfM)$  we see that the case $G=\Sp(\bfM)$ is universal in the sense that the object $\calA_{G,\bfM}$ in general should just be equal
to the $!$-pullback of $\calA_{\Sp(\bfM),\bfM}$.\footnote{This was first observed by V.~Drinfeld.}

In this paper we do the following:

1) We construct the object $\calA_{G,\bfM}$  (as was explained above it is enough to do it in the case $G=\Sp(\bfM)$).

2) We check that when $\bfM=T^*\bfN$ for some representation $\bfN$ of $G$,
this construction coincides with the one of~\cite{bfnc}.

3) In the case when $G=\Sp(\bfM)$ we compute the image of $\calA_{G,\bfM}$ under the twisted version
of the derived geometric Satake equivalence (see~\S\ref{1.8} below). To do that we express
$\calA_{G,\bfM}$ as a Radon transform of a certain theta-sheaf~\cite{l,ll} for the curve $\PP^1$
(the necessary facts and definitions about the Radon transform are reviewed in Appendix~\ref{gd}).
The idea that $\calA_{G,\bfM}$ should be related to the theta-sheaf also belongs to V.~Drinfeld.

\subsection{Idea of the construction}
Let us briefly explain the idea of the construction of $\calA_{G,\bfM}$. Let $\calC$ be a (dg) category endowed with a strong action of an algebraic group $H$ (e.g.\ one can take $\calC$ to be the (dg-model of the) derived category of $D$-modules on a scheme $X$ endowed with an action of $H$). Let $\calF$ be an object of $\calC$ which is equivariant under some closed subgroup $L$ of $H$. Then one can canonically attach to $\calF$ a ring object $\calA_\CF\in\mathrm{D}\modu^L(H/L)$ (the $L$-equivariant derived category of $D$-modules on $H/L$; this category is endowed with a natural monoidal structure with respect to convolution). This object has the property that its !-restriction to any $h\in H$ is equal to $\RHom(\calF,\calF^h)$.

Here is a variant of this construction. Assume that $H$ is endowed with a central extension
\[1\to \BC^\times\to \widetilde{H}\to H\to 1,\]
which splits over $L$. Then for any $\kappa\in \CC$ it makes sense to talk about an action $H$ on $\calC$ of level $\kappa$. Then in the same way as above we can define $\calA_{\calF}\in \DMOD^L(H/L)$ (here $\DMOD$ stands for the corresponding category of twisted $D$-modules on $H/L$).
The same thing works when $H$ is a group ind-scheme. We are going to apply it to the case when
$H=G_{\calK}, L=G_{\calO}, \calC=\CW\modu$, where $\CW$ is the Weyl algebra of the symplectic vector
space $\bM_\CK$. The line bundle $\CL_{\bfM}$ defines a central extension
${\widetilde G}_{\calK}$ of $G_{\calK}$, and it is well-known that the action of $G_{\calK}$ on
$\bM_\CK$ naturally extends to a strong action of ${\widetilde G}_{\calK}$ on $\CW\modu$ of
level $-1/2$.\footnote{This action should be thought of as
a categorical analog of the Weil representation, cf.~\cite{ll} and~\cite{la2}.}
We now take $\calF$ to be $\BC[\bM_\CO]$. The corresponding ring object $\calA_{G,\bfM}$
is just (the Riemann-Hilbert functor applied to) $\calA_\calF$ for $\calF$ as
above. It is not difficult
to check that when $\bfM=T^*\bfN$ this construction coincides with the one of \cite{bfnc}.

\begin{rem}
Here we make a remark about a connection between the above construction and some physics terminology.
Suppose $\bfM$ is a symplectic representation of $G$ and suppose the anomaly cancellation holds. In this case, physicists would say that there are two (closely related) structures attached to this data:

\medskip
a) a 3d $\calN = 4$ theory $T(G,\bfM)$ such that $T(G,\bfM)$ has what physicists call $G$-flavor symmetry.
In this case one can {\em gauge this symmetry} to get a new 3d $\calN=4$ theory; this new theory is the theory
$\calT(G,\bfM)$ discussed in~\S\ref{3dN4};

b) a supersymmetric boundary condition $\calB(G,\bfM)$ for 4d $\calN = 4$ Yang-Mills.

\medskip
\noindent
The relationship between the two is that $\calT(G,\bfM)$ is obtained from $\calB(G,\bfM)$ by pairing with the Dirichlet boundary condition for Yang-Mills; this implies that $T(G,\bfM)$ has $G$-flavor symmetry (it comes from the corresponding symmetry of the Dirichlet boundary condition).
Our constructions yield algebraic data attached to \emph{A-twists} of the resulting physical theories. The category $\CW$-mod is the \emph{category of line operators} of (the A-twist of) $T(G,\bfM)$, and the $G_{\calK}$-action on $\CW$-mod expresses the $G$-flavor symmetry of $T(G,\bfM)$. More details about the connection between our language and the physics language can be found in \cite{hr}.
\end{rem}

\subsection{$S$-duality and Ben-Zvi-Sakellaridis-Venkatesh conjectures}
This subsection is somewhat digressive from the point of view of the main body of this paper. We include it here for completeness and in order to indicate some future research directions.
\subsubsection{$S$-duality for boundary conditions}
The papers \cite{gw1,gw2} developed the theory of super-symmetric boundary conditions in 4d gauge
theories; it follows from {\em loc.cit.}\ that in addition to symplectic duality one should expect
some kind of $S$-duality for affine symplectic varieties $\bfM$ endowed with a Hamiltonian action of $G$ (here we no longer assume that $\bfM$ is a vector space) and with a $\BC^\times$-action for which the symplectic form has degree 2 --- again, satisfying some kind of anomaly cancellation condition (we don't know how to formulate it precisely, but when $\bfM$ is a symplectic  vector space with a linear action of $G$, it should be the same condition as before;  also, this condition should automatically be satisfied when $\bfM=T^*\bfN$ where $\bfN$ is a smooth affine $G$-variety). The $S$-dual of $\bfM$ is another affine variety $\bfM^{\vee}$ endowed with a Hamiltonian action of the Langlands dual group $G^{\vee}$. In fact, this kind of duality is not expected to be well-defined for arbitrary $\bfM$ --- only in some ``nice'' cases, which we don't know how to describe mathematically. Physically, it is explained in
{\em loc.cit.}\ that to any $\bfM$ as above one can attach a super-symmetric boundary condition in the corresponding 4-dimensional gauge theory; $S$-duality is supposed to be a well-defined operation on such boundary conditions, but since not all super-symmetric boundary conditions come from $\bfM$ as above, it follows that $\bfM^{\vee}$ will be well-defined only if we are sufficiently lucky.  It should also be noted that in general one should definitely consider singular symplectic varieties. On the other hand, below we describe a rather general construction and some expected properties of it. Let us also note that more generally, when the anomaly cancellation condition is not satisfied, one should expect a duality between varieties $\bfM$ and $\bfM^{\vee}$
endowed with some additional ``twisting data".

\subsubsection{The Whittaker reduction}
Before we discuss a somewhat general approach to the construction of the $S$-duality, let us give some explicit examples as well as some properties of $S$-duality. First we need to recall the notion of
Whittaker reduction.

Let $\bfM$ be any  Hamiltonian $G$-variety (i.e.\ $\bfM$ is a Poisson variety with a Hamiltonian $G$-action).
Let $\mu\colon \bfM\to \grg^*$ be the corresponding moment map.
Let also $U\subset G$ be a maximal unipotent subgroup of $G$ and let $\psi\colon U\to \GG_a$ be a generic homomorphism.
Then we set $\Whit_G(\bfM)$ to be the Hamiltonian reduction of $\bfM$ with respect to $(U,\psi)$. In other words, let us view $\psi$ as an element of $\gru^*$ (here $\gru$ is the Lie algebra of $U$) and let
$\grg^*_{\psi}$ be the pre-image of $\psi$ under the natural projection $\grg^*\to \gru^*$. Then
\[\Whit_G(\bfM)=(\mu^{-1}(\grg^*_{\psi}))/U.\]
It is well-known (cf.~\cite{kos}) that the action of $U$ on $\grg^*_{\psi}$ is free, so it is also free on $\mu^{-1}(\grg^*_{\psi})$.
Also note that $\mu^{-1}(\grg^*_{\psi})$ is an honest scheme (as opposed to a dg-scheme) because the moment map $\Whit_G(T^*G)  = \grg^*_\psi \overset U\times G \to \grg^*$ is smooth.

More generally, we can talk about the Whittaker reduction of any $G$-equivariant $\Sym(\grg)$-module. The connection between the Whittaker reduction and the derived Satake isomorphism is this: it is shown in~\cite{bef}
that for any $\calF\in D_{G_\CO}(\Gr_G)$ we have
\begin{equation}\label{sat-whit}
H^\bullet_{G_\CO}(\Gr_G,\calF)=\Whit_{G^{\vee}}(\Phi(\calF)).
\end{equation}

\subsubsection{Some expected properties of $S$-duality}
\label{expected}
Here are some purely mathematical properties that are expected to be satisfied by the $S$-dual
variety $\bfM^{\vee}$ (when it is well-defined):

1) Assume that $\bfM$ is a point. Then $\bfM^{\vee}=\Whit_{G^{\vee}}(T^*G^{\vee})$ (note that $T^*G^{\vee}$ is endowed with two commuting $G^{\vee}$-actions, so after we take the Whittaker reduction with respect to one of them, the 2nd one remains).

2) Let $H$ be a connected reductive group and set $G=H\times H$. Let $\bfM=T^*H$ (with natural $G$-action). Then we should
have $\bfM^{\vee}=T^*H^{\vee}$.

3) Assume that $\bfM$ is a linear symplectic representation of $G$ satisfying the anomaly cancellation condition. Then one should have
\begin{equation}\label{coul-s}
\calM_C(G,\bfM)=\Whit_{G^{\vee}}(\bfM^{\vee}).
\end{equation}

4) We expect that $(\bfM^{\vee})^{\vee}=\bfM$ whenever it makes sense.

\subsubsection{Construction of $\bfM^{\vee}$ in the cotangent case}
Here is a construction in the case when $\bfM=T^*\bfN$ where $\bfN$ is a smooth affine $G$-variety.
The construction of the ring object $\calA_{G,\bM}$ from \cite{bfnc} makes sense verbatim in this case (in \cite{bfnc} $\bfN$ was a vector space but it is not important for the construction). Let us consider $\Phi_G(\calA_{G,\bM})$. This is a commutative ring object of the derived category of $G^{\vee}$-equivariant dg-modules over $\Sym^\bullet(\gvee[-2])$. Passing to its cohomology $H^\bullet(\Phi_G(\calA_{G,\bM}))$ we just get a graded commutative algebra over
$\Sym^\bullet(\gvee[-2])$.\footnote{In all the interesting cases we know the algebra
$\Phi_G(\calA_{G,\bM})$ is formal, so we do not loose any information after passing to cohomology.}
Assuming that it has no cohomology in odd degrees, we can pass to its spectrum $\bM^\vee$.
This is an affine scheme with an action of $G^{\vee}$ which is endowed with a compatible map to $(\grg^\vee)^*$. In fact, the object $\calA_{G,\bM}$ is naturally equivariant with respect to the $\BC^\times$-action which rescales $t\in \calK$ (this action is usually
called ``loop rotation"). It is not difficult to see that (in the same way as in \cite{bfnc}) this defines a natural non-commutative deformation of the ring $H^\bullet(\Phi_G(\calA_{G,\bM}))$, and it particular, we get a Poisson structure on $\bfM^{\vee}$. This Poisson structure is easily seen to be generically symplectic and the above map to $(\gvee)^*$ is the moment map for the $G^{\vee}$-action and this Poisson structure.
The grading on the ring $H^\bullet(\Phi_G(\calA_{G,\bM}))$ defines a $\BC^\times$-action on $\bfM^{\vee}$ with respect to which the symplectic form has degree~2 (more precisely, we must divide the homological grading by~2: we can do that since we are assuming that we only have cohomology in even degrees).

It is easy to see that the above definition satisfies properties~1-3) of~\S\ref{expected}.
Namely,~1) is proved in~\cite{bef},~2) essentially follows from the construction of the derived
Satake equivalence, and~3) immediately follows from~\eqref{sat-whit}. On the other hand, property~4)
does not hold in this generality --- it fails already when $G$ is trivial; in general it is hard to
formulate since typically even if $\bfM=T^*\bfN$ with smooth $\bfN$, the variety $\bfM^{\vee}$ will
be singular; also if it is smooth it might not be isomorphic to a cotangent of anything. But even
when it is, the involutivity of the duality is far from obvious. Again, we believe that in
some ``nice'' cases the equality $(\bfM^{\vee})^{\vee}=\bfM$ makes sense and it is true (we do not
know how to say what ``nice'' means, but some examples are discussed below).

One can construct a natural functor from $D_{G_\CO}(\bN_\CK)$ to
$D^{G^\vee}(\Phi_G(\calA_{G,\bM}))$. Assuming formality of the ring $\Phi_G(\calA_{G,\bM})$ we can just think about
the latter category as the derived category of $G^{\vee}$-equivariant dg-modules over the coordinate ring $\CC[\bfM^{\vee}]$, when the latter is regarded as a dg-algebra with trivial differential and grading given by the above
$\BC^\times$-action. Ben-Zvi, Sakellaridis and Venkatesh conjectured that when $\bfN$ is a spherical variety for $G$ (i.e.\ when it has an open orbit with respect to a Borel subgroup of $G$), this functor is an equivalence. In fact, in this formulation the above conjecture is not very hard --
the real content of the conjecture (which we are not going to describe here) is hidden in the explicit (essentially combinatorial) calculation of $\bfM^{\vee}$ when $\bfM=T^*\bfN$, where $\bfN$ is a smooth spherical $G$-variety (this is done in \cite{bzsv}; also, under some assumptions the conjecture of \cite{bzsv} should hold for singular spherical $\bfN$, but in this case it is much harder to formulate).

\subsubsection{An example}\label{gln}
Here is another example. Let $G=\GL(N)\times \GL(N-1)$ and let $\bfM=T^*\GL(N)$ where the action of $G$ comes from the action of $\GL(N)$ on itself by left multiplication and from the action of $\GL(N-1)$  by right multiplication via the standard embedding $\GL(N-1)\hookrightarrow \GL(N)$. In this case $\bfN=\GL(N)$ is a spherical $G$-variety. Then it is essentially proved in~\cite{bfgt} that $\bfM^{\vee}=T^*\Hom(\CC^N,\CC^{N-1})$ and the Ben-Zvi-Sakellaridis-Venkatesh conjecture holds. It is, however, not clear how to deduce from this that $(\bfM^{\vee})^{\vee}=\bfM$. A construction of the isomorphism $(T^*\Hom(\CC^N,\CC^{N-1}))^{\vee}\simeq T^*\GL(N)$ is going to appear in a forthcoming paper of T.-H.~Chen and J.~Wang.

\subsubsection{$S$-duality outside of the cotangent type (linear case)}
In all of the above examples we only worked with cases when $\bfM=T^*\bfN$ for some smooth affine $G$-variety $\bfN$.
However, the main construction of this paper allows us to extend it to the case when $\bfM$ is an arbitrary symplectic representation of $G$ satisfying the anomaly cancellation condition.\footnote{One can also talk about $S$-duality for twisted objects, but we will not discuss it here.} Namely, as before we just let $\bfM^{\vee}$ be the spectrum of
$H^\bullet(\Phi_G(\calA_{G,\bfM}))$ (also as before let us assume that there is no cohomology in odd degrees).

The following example is similar to the one of~\S\ref{gln}. Let $N$ be a positive integer. Let $G=\Sp(2N)\times \SO(2N)$. Let also $\bfM$ be the bi-fundamental representation of $G$ (i.e.\ $\bfM=\CC^{2N}\otimes \CC^{2N}$ with the natural action of $G$).
Then $G^{\vee}=\SO(2N+1)\times \SO(2N)$, and we conjecture that $\bfM^{\vee}=T^*\SO(2N+1)$ (with the action of $G^{\vee}=\SO(2N+1)\times \SO(2N)$ defined similarly to the example in~\S\ref{gln}). Note that if $N>2$ then
$\bfM$ is an irreducible representation of $G$, so it cannot be written as $T^*\bfN$ for another representation $\bfN$.
On the other hand, $\bfM^{\vee}$ is manifestly written as a cotangent bundle to $\bfN^{\vee}=\SO(2N+1)$ and the fact that $(\bfM^{\vee})^{\vee}=\bfM$ (together with the corresponding special case of the Ben-Zvi-Sakellaridis-Venkatesh conjecture) is proved in~\cite{bft}. However, we do not know at the moment how to prove that $\bfM^{\vee}=T^*\SO(2N+1)$ (but at least the main construction of this paper allows us to formulate this statement).

Here is a variant of this example. Let $G=\SO(2N)\times \Sp(2N-2)$ (here we assume that $N>1$) and let $\bfM$ be again its bi-fundamental representation. Then $G^{\vee}=\SO(2N)\times \SO(2N-1)$,
and we expect that $\bfM^{\vee}=T^*\SO(2N)$ (the action
of $G^{\vee}=\SO(2N)\times \SO(2N-1)$ is again defined similarly to the example in~\S\ref{gln}).

\subsection{The universal ring object under Satake equivalence}
\label{1.8}
Finally, we are able to describe the image of the universal ring object under the twisted
Satake equivalence (answering a question of V.~Drinfeld). First, it turns out that for
$G=\Sp(\bM),\ \fg=\fsp(\bM)$, there is a monoidal equivalence
$\Phi_G\colon D_{-1/2}^{G_\CO}(\Gr_G)\iso D^G(\Sym^\bullet(\fg[-2]))$~\cite[Example 1.10]{dlyz}. Second,
$\Phi_G(\CA_{G,\bM})\cong\BC[\Whit_G(T^*G)]$ (Whittaker reduction of the shifted cotangent
bundle of $G$ with respect to the left action. The cohomological grading arises from the
one on $\BC[T^*G]=\BC[G]\otimes\Sym^\bullet(\fg)$, where the generators in $\fg$ are assigned
degree~2, while $\BC[G]$ is assigned degree~0).

Note that under the non-twisted Satake equivalence
$\Phi_{G^\vee}\colon D_{G^\vee_\CO}(\Gr_{G^\vee})\iso D^G(\Sym^\bullet(\fg[-2]))$, we have
$\Phi_{G^\vee}(\bomega_{\Gr_{G^\vee}})\simeq\Phi_G(\CA_{G,\bM})$. This answer to Drinfeld's question
was proposed by D.~Gaiotto.

Also, if we consider $G^\vee\cong\SO(\bM')$ for a $2n+1$-dimensional vector space $\bM'$ equipped
with a nondegenerate symmetric bilinear form, then $\bM\otimes\bM'$ carries a natural symplectic
form and a natural action of $G\times G^\vee$. We have an isomorphism
$\Phi_G(\CA_{G,\bM})\cong\BC[\Whit_{G^\vee}(\bM\otimes\bM')]$ (with residual action of $G$.
The cohomological grading arises from the one on $\Sym^\bullet(\bM\otimes\bM')$ where all the
generators are assigned degree~1).

Similarly, in the universal cotangent case, when $G=\GL(\bN)$ for an $n$-dimensional vector
space $\bN$, and $G^\vee\cong\GL(\bN')$ for another $n$-dimensional vector space $\bN'$,
we have the untwisted Satake equivalence
$\Phi_G\colon D_{G_\CO}(\Gr_G)\iso D^G(\Sym^\bullet(\fgl(\bN)[-2]))$.
Now $\Hom(\bN,\bN')\oplus\Hom(\bN',\bN)$ carries a natural sympectic form and a natural action
of $G\times G^\vee$. We have an isomorphism
$\Phi_G(\CA_{G,\bN})\cong\BC\Big[\Whit_{G^\vee}\big(\Hom(\bN,\bN')\oplus\Hom(\bN',\bN)\big)\Big]$
(with residual action of $G$. The cohomological grading arises from the one on
$\Sym^\bullet\big(\Hom(\bN,\bN')\oplus\Hom(\bN',\bN)\big)$ where all the generators are assigned
degree~1).

\subsection{Acknowledgments} We are deeply grateful to D.~Ben-Zvi, R.~Bez\-ru\-kavnikov, V.~Drinfeld,
P.~Etingof, B.~Feigin, D.~Gaiotto, D.~Gaitsgory, A.~Hanany, T.~Johnson-Freyd, S.~Lysenko, H.~Nakajima,
Y.~Sakellaridis, C.~Teleman, A.~Venkatesh, J.~Wang, E.~Witten, P.~Yoo, Z.~Yun and Y.~Zhao for many
helpful and inspiring discussions.
M.F.\ and S.R.\ thank the
4th Nisyros Conference on Automorphic Representations and Related Topics held in July 2019 for
stimulating much of this work.
We are also obliged to the referee for his careful reading of the manuscript and numerous useful
suggestions and corrections.

A.B.~was partially supported by NSERC.
G.D.~was supported by an NSF Postdoctoral Fellowship under grant No.~2103387.
The research of M.F.~was supported by the Israel Science Foundation (grant No.~994/24).
S.R.~was supported by the NSF grant DMS-2101984.

\section{Setup and notation}

\subsection{Generalities}
\label{generalities}
Let $\CK=\BC\dprts t\supset\CO=\BC\dbkts t$. Let $\fG$ be a complex reductive group.
The affine Grassmannian ind-scheme $\Gr_\fG=\fG_\CK/\fG_\CO$
is the moduli space of $\fG$-bundles on the formal disc equipped with a trivialization on the
punctured formal disc. Equivalently, it is the moduli space of $\fG$-bundles on a smooth curve $C$
equipped with a trivialization off a point $x\in C$. More generally, the Beilinson-Drinfeld
affine Grassmannian $\Gr_{\fG,BD,C^I}$ is the moduli spaces of maps $I\to C$ from a finite set $I$
to $C$, and $\fG$-bundles on $C$ equipped with a trivialization off the image of $I$.
A surjection $\varphi\colon I\twoheadrightarrow J$ gives rise to a natural closed embedding
$\varphi\colon\Gr_{\fG,BD,C^J}\hookrightarrow\Gr_{\fG,BD,C^I}$.

A {\em factorization line bundle} $\CL$ on $\Gr_\fG$ is a collection of line bundles
$\CL_{C^I}$ on $\Gr_{\fG,BD,C^I}$ equipped with isomorphisms $\varphi^*\CL_{C^I}\simeq\CL_{C^J}$.
In particular, such a line bundle restricted to the marked point $x\in C$ gives rise to a line
bundle on $\Gr_\fG$ multiplicative with respect to the convolution.
The group of isomorphism classes of factorization line bundles (with respect to tensor product)
is denoted $\on{Pic}_{\on{fact},C}(\Gr_\fG)$. In case $C$ is the affine line $\BA^1$,
the group $\on{Pic}_{\on{fact},\BA^1}(\Gr_\fG)$ 
is canonically isomorphic to $H^2_{\on{Zar}}(\on{B}\!\fG,\ul{K}{}_2)\cong H^4(\on{B}\!\fG,\BZ(2))$
(motivic cohomology of the classifying space of $\fG$), that in turn is canonically isomorphic
to the group $\on{Quad}(X_*(T))^W$ of {\em even-valued} quadratic forms on the coweight lattice
$X_*(T)$ of a Cartan torus $T\subset\fG$, invariant with respect to the Weyl group $W(\fG,T)$
(see e.g.~\cite[\S4.2.5]{ga} or~\cite[\S2.3]{z}). One can also consider the group
$\on{SPic}_{\on{fact},\BA^1}(\Gr_\fG)$ of isomorphism classes of factorization {\em super} line bundles
(compatible collections of $\BZ/2\BZ$-graded line bundles\footnote{see e.g.~\cite[Exemple 4.1]{d}}
on $\Gr_{\fG,\BD,C^I}$). It is canonically isomorphic to the group $\on{Bil}(X_*(T))^W$ of integer-valued
$W$-invariant symmetric bilinear forms $B$ on $X_*(T)$ such that $B(\lambda,\lambda)\in2\BZ$ for any $\lambda$
in the coroot lattice $Q\subset X_*(T)$~\cite[Theorem B]{z2}.
The parity of such a line bundle on a connected component of $\Gr_\fG$ is equal
to the parity of $B(\mu,\mu)$ for any coweight $\mu$ representing a $T$-fixed point in this
component.

A factorization line bundle is automatically multiplicative with respect to the convolution
on $\Gr_\fG$. Indeed, by the above classification of factorization line bundles, any such $\CL$
gives rise to a central extension of $\fG$ by $K_2$. This extension in its turn gives rise
to a central extension of $\fG(\CK)$ by $\BG_m$ canonically split over $\fG(\CO)$,
see~\cite[\S4.3]{kv}. The latter extension is nothing but a multiplicative (with respect to
the convolution) structure on $\CL$.

Given $c\in\BC$ and a factorization super line bundle $\CL$ on $\Gr_\fG$ we denote by
$D^{\fG_\CO}(\on{D-mod}_{\CL^c}(\Gr_\fG))$ the $\fG_\CO$-equivariant derived category of
$D$-modules on $\Gr_\fG$, twisted by $\CL^c$. 
This category is endowed with a convolution tensor product compatible with factorization
(which automatically gives rise a commutativity constraint (in the super-sense) on this category).

\subsection{Symplectic affine Grassmannians}
\label{aff gr}
Let $\bM$ be a $2n$-dimensional complex vector space equipped with a symplectic form
$\langle\, ,\rangle$. Its automorphism group is $G=\Sp(\bM)$.

The Kashiwara affine Grassmannian infinite type scheme
$\GR_G=G_\CK/G_{\BC[t^{-1}]}$ is the moduli space of $G$-bundles on $\BP^1$ equipped with a
trivialization in the formal neighbourhood of $0\in\BP^1$.

The determinant line bundles over $\Gr_G$ and $\GR_G$ are denoted by $\CalD$.
The $\mu_2$-gerbe of square roots of $\CalD$ over $\GR_G$ (resp.\ $\Gr_G$) is denoted $\tGr_G$
(resp.\ $\tGR_G$).

The action of $G_\CK$ on $\Gr_G$ and $\GR_G$ lifts to the action of the metaplectic group-stack
$\tG_\CK$ on $\tGr_G$ and $\tGR_G$. We have a splitting $G_\CO\hookrightarrow\tG_\CK$.

In what follows we only consider the {\em genuine} constructible sheaves on $\tGr_G$ and
$\tGR_G$: such that $-1\in\mu_2$ acts on them as $-1$.
We consider a dg-enhancement $D^b_{G_\CO}(\tGr_G)$ of the (genuine) bounded equivariant
constructible derived category. We denote by $D_{G_\CO}(\tGr_G)$ the {\em renormalized} equivariant
derived category defined as in~\cite[\S12.2.3]{arga}.
We also consider the category $D_{G_\CO}(\tGR_G)_!$ defined as in~\cite[\S3.4.1]{ag} (the inverse
limit over the $G_\CO$-stable open subgerbes of $\tGR_G$, cf.~\S\ref{A4}). It contains
the IC-sheaves of the $G_\CO$-orbits closures.

An open sub-gerbe $\CT\hookrightarrow \tGr_G\times\tGR_G$ is formed by all the pairs of transversal
compact and discrete Lagrangian subspaces in $\bM_\CK$. We denote by
\[\tGr_G\xleftarrow{p}\CT\xrightarrow{q}\tGR_G\] the natural projections.
The {\em Radon Transform} is (cf.~\S\ref{A5}, where its $D$-module version is denoted $\RT_!^{-1}$)
\begin{equation}
  \label{radon}
  \RT:=p_*q^!\colon D_{G_\CO}(\tGR_G)_!\to D_{G_\CO}(\tGr_G).
\end{equation}

The {\em Theta-sheaf} $\Theta\in D_{G_\CO}(\tGR_G)_!$ introduced in~\cite{l} is the direct sum of IC-sheaves
of two $G_\CO$-orbits in $\tGR_G\colon \Theta_g$ of the open orbit, and $\Theta_s$ of the codimension~1
orbit.

\subsection{D-modules}
\label{dmod}
The dg-category of $G_\CO$-equivariant $D$-modules on $\Gr_G$ (resp.\ on $\GR_G$) twisted by the
inverse square root $\CalD^{-1/2}$ is denoted $\Dmo(\Gr_G)^{G_\CO}$ (resp.\ $\Dmo(\GR_G)_!^{G_\CO}$).
More precisely, by $\Dmo(\Gr_G)^{G_\CO}$ we mean the {\em renormalized} equivariant
category defined as in~\cite[\S12.2.3]{arga}, and $\Dmo(\GR_G)_!^{G_\CO}$ is defined in~\S\ref{A4}.
We have the Riemann--Hilbert
equivalences \[\RH\colon\Dmo(\Gr_G)^{G_\CO}\iso D_{G_\CO}(\tGr_G),\ \Dmo(\GR_G)_!^{G_\CO}\iso D_{G_\CO}(\tGR_G)_!.\]
We denote $\RH^{-1}(\Theta)$ by $\varTheta\in\Dmo(\GR_G)_!^{G_\CO}$, a direct sum of two irreducible
$D$-modules, $\varTheta_g$ with the full support, and $\varTheta_s$ supported at the Schubert
divisor.

The (derived) global sections $\bGamma(\GR_G,\varTheta_g)$ and $\bGamma(\GR_G,\varTheta_s)$ are
irreducible $G_\CO$-integrable $\fg_\aff$-modules of central charge $-1/2$, namely $L^0_{-1/2}$ and
$L^{\omega_1}_{-1/2}$~\cite[Theorem 4.8.1]{kt}.
Here $\fg=\fsp(\bM)$, and the highest component of $L^0_{-1/2}$ (resp.\ $L^{\omega_1}_{-1/2}$) with respect
to $\fg_\CO$ is the trivial (resp.\ defining) representation of $\fg$.\footnote{For a finite dimensional
counterpart of this statement (about global sections of irreducible equivariant $D$-modules on the
Lagrangian Grassmannian of $\fg$), see~\S\ref{baby}.}

The (derived) global sections functors \[\Gamma\colon\Dmo(\Gr_G)^{G_\CO}\to\Rep^{G_\CO}_{-1/2}(\fg_\aff),\
\bGamma\colon\Dmo(\GR_G)_!^{G_\CO}\to\Rep^{G_\CO}_{-1/2}(\fg_\aff)\] ($G_\CO$-integrable $\fg_\aff$-modules with central charge
$-1/2$) admit the left adjoints (see~\S\S\ref{A7},\ref{A4})
\[\Loc\colon\Rep^{G_\CO}_{-1/2}(\fg_\aff)\to\Dmo(\Gr_G)^{G_\CO},\
\LOC\colon\Rep^{G_\CO}_{-1/2}(\fg_\aff)\to\Dmo(\GR_G)_!^{G_\CO}.\]
According to~\cite[Theorem 4.8.1(iv)]{kt}, we have
$\tau_{\geq0}\LOC(L^0_{-1/2}\oplus L^{\omega_1}_{-1/2})=\varTheta$ (the top cohomology in the natural
$t$-structure).

\subsection{Weyl algebra}
\label{weyl}
The symplectic form on $\bM$ extends to the same named $\BC$-valued symplectic form on
$\bM_\CK\colon \langle f,g\rangle=\on{Res}\langle f,g\rangle_\CK dt$. We denote by $\CW$ the completion
of the Weyl algebra of $(\bM_\CK,\langle\, ,\rangle)$ with respect to the left ideals generated by
the compact subspaces of $\bM_\CK$. It has an irreducible representation $\BC[\bM_\CO]$.
Also, there is a homomorphism of Lie algebras $\fg_\aff\to\on{Lie}\CW$, see e.g.~\cite{ff}.
According to~\cite[rows 3,4 of Table XII at page 168]{ff}, the restriction of $\BC[\bM_\CO]$ to
$\fg_\aff$ is $L^0_{-1/2}\oplus L^{\omega_1}_{-1/2}$ (even and odd functions, respectively).
\footnote{For
a finite dimensional counterpart of this statement (about restriction to $\fg$ of an irreducible module
over the Weyl algebra of $\bM$), see~\S\ref{baby}.}

We consider the dg-category $\CW\modu$ of discrete $\CW$-modules. We recall this is a renormalization of the naive derived category $\CW\modu^{\on{naive}}$ of discrete $\CW$ modules, or more carefully its canonical dg-enhancement, defined as follows. 

For each compact open subspace $\mathbf{K} \subset \mathbf{M}_{\CK}$, consider the module $V_{\mathbf{K}}$ obtained as the quotient of $\CW$ by the left ideal generated by $\mathbf{K}$. Let us denote by $\mathcal{E}$ the pre-triangulated envelope of all such modules $V_{\mathbf{K}}$ within $\CW\modu^{\on{naive}}$. By definition, $\CW\modu$ is the ind-completion of $\mathcal{E}$. It carries a unique $t$-structure for which the natural map 
$$\CW\modu \rightarrow \CW\modu^{\on{naive}}$$is $t$-exact.

 More concretely, we may identify
$\CW$ with the ring of differential operators on a Lagrangian discrete lattice $\bL\subset\bM_\CK$,
e.g.\ $\bL=t^{-1}\bM_{\BC[t^{-1}]}$. Then $\CW\modu$ is the inverse limit of ${\rm D}\modu(V)$ over
finite dimensional subspaces $V\subset\bL$ with respect to the functors $i_{V\hookrightarrow V'}^!$.
Equivalently, $\CW\modu$ is the colimit, in the sense of cocomplete dg-categories, of ${\rm D}\modu(V)$ with respect to the functors
$i_{V\hookrightarrow V',*}$. 

\begin{lem}There is a categorical action \[\Dmo(G_\CK)\circlearrowright\CW\modu.\] In particular, upon taking spherical vectors, there is an action \[\Dmo(\Gr_G)^{G_\CO}\circlearrowright(\CW\modu)^{G_\CO}.\]

\end{lem}

\begin{proof} We will obtain this from ~\cite[\S10]{r}. More precisely, we will
  apply~\cite[Corollary 10.23.3]{r} to obtain the desired action. To see that the hypotheses of
  the corollary are satisfied, following~\cite[Remark 10.23.4]{r} it is enough to notice that
  the compact objects in $\CW\modu$ are closed under truncation functors. However, the latter
  claim is visible from the Lagrangian picture (note the analogous claim is true for D-modules
  on any placid ind-scheme admitting a dimension theory).  
\end{proof}

\subsection{Twisted derived Satake}
\label{twisat}
According to~\cite[Example 1.10]{dlyz}, there is a monoidal equivalence
$\Phi\colon D^b_{G_\CO}(\tGr_G)\iso D^G_\perf(\Sym^\bullet(\fg[-2]))$
(dg-category of perfect complexes of dg-modules over the dg-algebra $\Sym^\bullet(\fg[-2])$
equipped with a trivial differential). It extends to a monoidal equivalence of Ind-completions
$\Phi\colon D_{G_\CO}(\tGr_G)\iso D^G(\Sym^\bullet(\fg[-2]))$.

Here is one of the key properties of the twisted derived Satake equivalence $\Phi$.
We choose a pair of opposite maximal unipotent subgroups $U_G,U_G^-\subset G$, their regular
characters $\psi,\psi^-$, and denote by
$\varkappa\colon D^G(\Sym^\bullet(\fg[-2]))\to D(\BC[\Xi_\fg])$ the functor of Kostant-Whittaker
reduction with respect to $(U_G^-,\psi^-)$
(see e.g.~\cite[\S2]{bef}). Here $\Xi_\fg$ with grading disregarded is the tangent bundle $T\Sigma_\fg$
of the Kostant slice $\Sigma_\fg\subset\fg^*$. Let us write $\kappa$ for the
Ad-invariant bilinear form on $\fg$, i.e., level, corresponding to our central
charge of $-1/2$. Explicitly, if we write $\kappa_b$ for the basic level giving
the short coroots of $\fg$ squared length two,
and $\kappa_c$ for the critical level,
then  $\kappa$ is defined by
\[
\kappa = - 1/2 \cdot \kappa_b-\kappa_c.
\]
If we consider the Langlands dual Lie algebra $\gvee\simeq\fso_{2n+1}$, the
form $\kappa$  gives rise to identifications $\Sigma_\fg\cong\Sigma_{\gvee}$ and
$\Xi_\fg\cong\Xi_{\gvee}$. Also, we have a canonical isomorphism
$H^\bullet_{G_\CO}(\Gr_G)\cong\BC[\Xi_{\gvee}]\cong\BC[\Xi_\fg]$. This is a theorem of
V.~Ginzburg~\cite{g} (for a published account see e.g.~\cite[Theorem 1]{bef}).

Now given $\CF\in D^b_{G_\CO}(\tGr)$ we consider the tensor product $\CF
\overset ! \otimes \on{RT}(\Theta)$
(notation of~\S\ref{aff gr}).
Since the monodromies of the factors cancel out, it canonically descends to
an object of $D_{G_\CO}(\Gr_G).$
The aforementioned key property is a canonical isomorphism
\begin{equation}
  \label{glob coh}
  H^\bullet_{G_\CO}(\Gr_G,\CF \overset !\otimes\on{RT}(\Theta))\cong\varkappa\Phi\CF
\end{equation}
of $H^\bullet_{G_\CO}(\Gr_G)\cong\BC[\Xi_\fg]$-modules.

\section{The universal ring object}

\subsection{The internal Hom construction} To introduce the universal ring
object and show its relation to the $\varTheta$-sheaf, we recall the following
general construction of internal Hom objects.

Let $\calC$ be a module category over $\Dmo(G_\CK)$. Given a subgroup $H$ of $G_\CK$ and an
\mbox{$H$-equivariant} object $\xi$ of $\calC$, convolution with it yields a
$\Dmo(G_\CK)$-equivariant functor $(\Dmo(G_\CK)_{*})_H \rightarrow \calC$, and
upon restriction to spherical vectors a $\Dmo(\Gr_G)^{G_\CO}$-equivariant
functor $\Dmo(G_\CO \backslash G_\CK)_{H} \rightarrow (\calC)^{G_\CO}.$ If both
$\calC$ and $(\Dmo(G_\CO \backslash G_\CK)_*)_H$ are dualizable as abstract
dg-categories, we  obtain the dual $\Dmo(\Gr_G)^{G_\CO}$-equivariant functor
\[
(\calC^\vee)^{G_\CO} \rightarrow \Dmo(G_\CK / H)^{G_\CO}_!, \ \zeta \mapsto
\CHom(\xi, \zeta).
\]

We apply this as follows. First, taking $\calC = \CW\mod$, $H =  G_{\CO}$, and
$\xi = \CC[\bM_{\CO}]$, we obtain a functor
\[F\colon (\CW \modu)^{G_\CO}\to\Dmo(\Gr_G)^{G_\CO},\
M\mapsto\CHom(\BC[\bM_\CO],M).\]
Setting $M = \CC[\bM_\CO]$, we obtain the internal Hom ring object
\[\fR:=\CHom(\BC[\bM_\CO],\BC[\bM_\CO])\in\Dmo(\Gr_G)^{G_\CO}.\]

Second, taking $\calC = \CW\mod$, $H = G_{\CC[t^{-1}]}$, and $\xi=
\bomega_{t^{-1} \bM_{\mathbb{C}[t^{-1}]}}$, i.e., the colimit of the dualizing
sheaves $\bomega_V$ over finite dimensional subspaces $V \subset t^{-1}
\bM_{\CC[t^{-1}]}$, we obtain a functor
\[\bF\colon (\CW\modu)^{G_\CO} \to \Dmo(\GR_G)_!^{G_\CO},\ M\mapsto
\CHom(\bomega_{t^{-1} \bM_{\mathbb{C}[t^{-1}]}},M).\]

\begin{lem}
	\label{laff}
	We have a canonical isomorphism $\bF(\BC[\bM_\CO])\cong\varTheta$.
\end{lem}

\begin{proof}
We have $\BC[\bM_\CO]=\CW/(\CW\cdot\bM_\CO)$. We denote $\bF(\BC[\bM_\CO])$ by $\CF$ for short.
For a Lagrangian discrete lattice $\bL$ representing a point of $\GR_G$, the fiber $\CF_\bL$
of $\CF$ at $\bL$ is $\CW/(\CW\cdot\bM_\CO+\bL\cdot\CW)$. According to~\cite[\S2]{la},
the fiber $\varTheta_\bL$ is $\CW/(\CW\cdot\bM_\CO+\bL\cdot\CW)$ as well.

For the reader's convenience, let us briefly sketch a proof of the latter isomorphism.
First, we consider the finite dimensional counterpart $\CS=\CS_g\oplus\CS_s$ of $\varTheta$
as in~\S\ref{baby}. For a Lagrangian subspace $L\subset\bM$ representing a point of $\LGr_\bM$,
the fiber $\CS_L$ of $\CS$ at $L$ is $\CW_\bM/(\CW_\bM\cdot\bN+L\cdot\CW_\bM)$
(notation of~\S\ref{baby}). This follows from the De Rham counterpart of the integral
presentation~\cite[Proposition 5]{l} of $S$.

Second, representing $\bM_\CK$ as an ind-pro-limit of a growing family of finite dimensional
symplectic spaces $\bM'$, we can construct the Theta $D$-module $\varTheta_{\on{Sato}}$ on the
co-Sato Lagrangian Grassmannian $\GR_{\on{Sato}}$ of Lagrangian discrete lattices in $\bM_\CK$
as a certain limit of baby Theta $D$-modules $_{\bM'}\CS$ on $\LGr_{\bM'}$, see~\cite[\S6.5]{ll}.
The similar formula for the fibers of $\varTheta_{\on{Sato}}$ follows. Finally, we have an embedding
$\GR_G\hookrightarrow\GR_{\on{Sato}}$, and $\varTheta$ is the pullback of $\varTheta_{\on{Sato}}$
by~\cite[Theorem 3]{ll}. Hence the desired formula for the fibers of $\varTheta$.
\end{proof}

\subsection{Radon transform}
Recall the Radon transform~\eqref{radon}. We keep the same notation for its
$D$-module version
$\RT\colon \Dmo(\GR_G)_!^{G_\CO}\to\Dmo(\Gr_G)^{G_\CO}$. See the Appendix
starting from~\S\ref{A5},
where it is denoted $\RT_!^{-1}$.

\begin{prop}
	\label{coprinus}
	We have an isomorphism $\fR\simeq\RT\varTheta$.
\end{prop}

\begin{proof}By Lemma \ref{laff}, it suffices to show that the composition
\[(\CW\modu)^{G_\CO}\xrightarrow{\bF} \Dmo(\GR_G)^{G_\CO}_!
\xrightarrow{\on{RT}} \Dmo(\Gr_G)^{G_\CO}\]
	is $\Dmo(\Gr_G)^{G_\CO}$-equivariantly equivalent to  $F$. By dualizing the
	appearing functors, we equivalently must show that the composition
	\[
	\Dmo(\Gr_G)^{G_\CO} \xrightarrow{\on{RT}^\vee} \Dmo(\GR_G)^{G_\CO}_*
	\xrightarrow{ \mathbf{F}^\vee}  (\CW\mod)^{G_\CO}
	\]
	sends the delta function at the origin $\delta_e$ to $\CC[\bM_\CO]$.
	
	To show this, writing $\on{Av}^{G_\CO}_{!}$ for the partially defined left
	adjoint to the forgetful functor $(\CW\mod)^{G_\CO} \rightarrow \CW\mod$,
	we have the following.
	
	\begin{lem} \label{lcompgen}
		The category $(\CW\modu)^{G_\CO}$ is compactly generated by a single
		object $\on{Av}^{G_\CO}_!( \CC[\bM_\CO])$.
	\end{lem}
	
	\begin{proof}
		We have an equivalence
		$(\CW\modu)^{G_\CO}\simeq\rm{D}\modu(\on{Heis})^{G_\CO\ltimes\bM_\CO\times\BG_a,\chi}$,
		where $\on{Heis}$ is the Heisenberg central extension of $\bM_\CK$ with
		$\BG_a$ (canonically split
		after restriction to $\bM_\CO$), and $\chi$ is the character of
		$G_\CO\ltimes\bM_\CO\times\BG_a$
		obtained by composition of projection to $\BG_a$ and exponentiating.
		Indeed, the $\CW$-module
		$\BC[\bM_\CO]$ is strongly
		$(G_\CO\ltimes\bM_\CO\times\BG_a,\chi)$-equivariant, and so gives rise
		to a functor from
		$\rm{D}\modu(\on{Heis})^{G_\CO\ltimes\bM_\CO\times\BG_a,\chi}$ to
		$(\CW\modu)^{G_\CO}$ that
		is the desired equivalence.
		
		Now $\chi$ is non-trivial on the stabilizer of any point
		$m\in\on{Heis}\setminus(\bM_\CO\times\BG_a)$. Indeed, given a vector
		$m\in\bM_\CK$ with nontrivial
		polar part, we can find $g\in G_\CO$ such that $gm=m+m'$, where
		$m'\in\bM_\CO$ has nonzero
		$\on{Res}\langle m,m'\rangle_\CK$. So $\chi|_{\on{Stab}(m)}$ is
		nontrivial.
		
		Hence any object of
		$\rm{D}\modu(\on{Heis})^{G_\CO\ltimes\bM_\CO\times\BG_a,\chi}$ must be
		supported on $\bM_\CO\times\BG_a$. This yields an equivalence
		$
		(\CW\mod)^{G_\CO} \simeq \on{D-mod}( \on{pt}/G_\CO),
		$
		which exchanges $\CC[\bM_\CO]$ with the dualizing sheaf. Moreover, if
		we write $\langle \CC[\bM_\CO] \rangle$ for the full subcategory of
		$\CW\mod$ compactly generated by $\CC[\bM_\CO]$, this exchanges the
		forgetful functor
		\[
		(\CW\mod)^{G_\CO} \rightarrow \langle \CC[\bM_\CO] \rangle \simeq
		\on{Vect}
		\]
		with the functor of $!$-pullback to the point \[\on{D-mod}(\pt/G_\CO)
		\rightarrow \on{D-mod}(\pt) \simeq \on{Vect}.\]
                The claim of the lemma
		now follows from the analogous fact for D-modules on $\pt/G_\CO$,
		see for example~\cite[\S 7.2.2]{drga}. \end{proof}

	We are now ready to calculate $\bF^\vee \circ \RT^\vee( \delta_e)$. First, if
	we write $j_* \in \Dmo(\GR_G)_*^{G_\CO}$ for the $*$-extension of the
	constant D-module on the big cell, unwinding definitions we have that
	\begin{equation*}
	\bF^\vee \circ \RT^\vee(\delta_e) \simeq \bF^\vee ( j_* ) \simeq
	j_*  \star \bomega_{t^{-1} \bM_{\mathbb{C}[t^{-1}]}}.
	\end{equation*}
	To identify this with $\CC[\bM_\CO]$, by the proof of Lemma
	\ref{lcompgen},  particularly the exhibited equivalence $(\CW\mod)^{G_\CO}
	\simeq \on{D-mod}(\pt / G_\CO)$, we must show
	that
	$
	\Hom_{(\CW\mod)^{G_\CO}}(\on{Av}^{G_\CO}_!(\CC[\bM_\CO]), j_* \star
	\bomega_{t^{-1} \bM_{\mathbb{C}[t^{-1}]}})
	$
	is the trivial line $\CC$, placed in cohomological degree zero.
	
	To see this, note that $j_*$ identifies with the relative $*$-averaging
	$(\CW\mod)^G \to (\CW\mod)^{G_\CO}$, and that, by the prounipotence of the
	kernel of $G_\CO \to G$ and the $G_\CO$-equivariance of $\CC[\bM_\CO]$, one
	has a canonical equivalence $\on{Av}_!^{G}(\CC[\bM_\CO]) \simeq
	\on{Av}_!^{G_\CO}(\CC[\bM_\CO])$. Therefore, we may compute
	\begin{multline*}
	\Hom_{(\CW\mod)^{G_\CO}}(\on{Av}^{G_\CO}_!(\CC[\bM_\CO]), j_* \star
	\bomega_{t^{-1} \bM_{\mathbb{C}[t^{-1}]}}) \\  \simeq
	\Hom_{(\CW\mod)^{G}}(\on{Av}^{G_\CO}_!(\CC[\bM_\CO]),  \bomega_{t^{-1}
	\bM_{\mathbb{C}[t^{-1}]}})  \\  \simeq \Hom_{\CW\mod}(\CC[\bM_\CO],
	\bomega_{t^{-1} \bM_{\mathbb{C}[t^{-1}]}}) \\  \simeq
	\Hom_{\on{D-mod}(\bM_{t^{-1} \CC[t^{-1}]})}(\delta_0,
	\bomega_{t^{-1} \bM_{\mathbb{C}[t^{-1}]}}) \simeq \mathbb{C},
	\end{multline*}
	as desired. \end{proof}

\begin{cor}
  We have an isomorphism $\Gamma(\fR)\simeq\BC[\bM_\CO]$.
\end{cor}

\begin{proof}
   Recall that $\mathbf{\Gamma}(\varTheta) \simeq \CC[\bM_{\CO}]$ and apply
   Proposition~\ref{p:locpos}.
\end{proof}

\subsection{Costalks of $\fR$}
The $G_\CO$-orbits in $\Gr_G$ are numbered by the dominant coweights of $G$, i.e.\ collections
of integers $\lambda=(\lambda_1\geq\ldots\geq\lambda_n\geq0)$. Given such $\lambda$, we denote
by $\imath_\lambda$ (resp.\ $i_\lambda$) the embedding of the corresponding Cartan torus fixed point
(resp.\ of the locally closed orbit $\Gr_G^\lambda$) into $\Gr_G$.

\begin{lem}
  \label{costalks}
  The costalk $\imath_\lambda^!\fR$ is one-dimensional; it lives in cohomological degree
  $\lambda_1+\ldots+\lambda_n$. Equivalently, the corestriction $i_\lambda^!\fR$ is a one-dimensional
  (twisted) local system on $\Gr_G^\lambda$, living in cohomological degree 
  $\lambda_1+\ldots+\lambda_n-2\dim\Gr_G^\lambda$.
\end{lem}

\begin{proof}
  By definition, we have $\imath_\lambda^!\fR=\Hom_{\CW\mod}(\CC[\bM_\CO],\CC[t^\lambda(\bM_\CO)])$.
  The latter Hom is 1-dimensional sitting in the degree equal to the codimension of
  $\bM_\CO\cap t^\lambda(\bM_\CO)$ in $\bM_\CO$. This codimension is nothing but
  $\lambda_1+\ldots+\lambda_n$.
\end{proof}

\begin{rem}
  \label{integral}
Recall that the Theta-sheaf $\Theta=\on{RH}\varTheta\in D_{G_\CO}(\tGR_G)_!$ is the direct sum
of IC-sheaves of two $G_\CO$-orbits in $\tGR_G\colon \Theta_g$ of the open orbit, and $\Theta_s$
of the codimension~1 orbit. The local systems on these orbits admit unique integral structures
(local systems of $\BZ$-modules of rank~1). Hence their Goresky-MacPherson extensions acquire
integral structures, i.e.\ $\Theta$ carries a canonical integral structure (a monodromic perverse
sheaf of $\BZ$-modules). It induces an integral structure on the Radon transform
$\on{RH}\fR=\on{RT}(\Theta)$. In particular, any costalk of $\on{RH}\fR$ (at a point of the
punctured determinant line bundle $\bCD$) carries an integral structure, i.e. a signed basis.
\end{rem}

\subsection{Computation of $\RH\fR$ under the twisted derived Satake}
\label{computa}
Recall the notation of~\S\ref{twisat}. We consider an object
$\BC[G]\otimes\Sym^\bullet(\fg[-2])\in D^G(\Sym^\bullet(\fg[-2]))$. In fact,
$\BC[G]\otimes\Sym^\bullet(\fg[-2])$ has {\em two} such structures: with respect to the left
(resp.\ right) $G$-action and the left (resp.\ right) comoment morphism.
We consider the hamiltonian reduction with respect to the right $U_G$-action
\[\big(\BC[G]\otimes\Sym{}\!^\bullet(\fg[-2])\big)/\!\!/\!\!/(U_G,\psi_G).\]
This reduction has the residual left structure of a ring object of $D^G(\Sym^\bullet(\fg[-2]))$.
We will denote this object by $\fK$.

\begin{thm}
  \label{Phirka}
  We have an isomorphism of ring objects $\Phi\RH\fR\simeq\fK$.
\end{thm}

\begin{proof} An underlying equivalence of objects $\Phi \RH \fR \simeq \fK$, ignoring the ring structures, is supplied
  in~\cite[Proposition 9.54]{dlyz}.  

Now to compare the ring structures, recall that each object $K$ of $D^G(\Sym^\bullet(\fg[-2]))$
splits into direct sum $K=K_0\oplus K_1$ of its even and odd parts according to the action of
the center $\{\pm1\}\subset G$. In particular, $\fK=\fK_0\oplus\fK_1$. Let us denote the monoidal
structure on $D^G(\Sym^\bullet(\fg[-2]))$ by $\star$. Then $\Ext^\bullet(\fK\star\fK,\fK)$ lives in nonnegative
degrees, and $\Ext^0(\fK\star\fK,\fK)=\BC^4$. More precisely, \[\Ext^0(\fK_0\star\fK_0,\fK_0)\simeq\BC,\
\Ext^0(\fK_1\star\fK_1,\fK_0)\simeq\BC,\ \Ext^0(\fK_0\star\fK_1,\fK_1)\simeq\BC.\]
The ring structure on $\fK$ gives rise to a {\em nonzero} morphism in
$\Ext^0(\fK_1\star\fK_1,\fK_0)=\BC$, and is uniquely characterized by this nonvanishing.

So it remains to check that the internal Hom composition morphism $\fR_1\star\fR_1\to\fR_0$
does not vanish. To this end it is convenient to use the modified objects
$\widetilde\fR_0,\widetilde\fR_1$ of~\cite[\S2.9]{bft2} such that
$\widetilde\fR_0\star\widetilde\fR_0\simeq\widetilde\fR_0$, and 
$\widetilde\fR_0\star\widetilde\fR_1\simeq\widetilde\fR_1$. We must check that the internal Hom
composition morphism $\widetilde\fR_1\star\widetilde\fR_1\to\widetilde\fR_0$ does not vanish.
This follows from the fact that $\BC[\bM_\CO]$ generates the dg-category $\CW\modu^{G_\CO}$,
and $\widetilde\fR_0\star\BC[\bM_\CO]\simeq\BC[\bM_\CO]\simeq\widetilde\fR_1\star\BC[\bM_\CO]$.
\end{proof}

\begin{rem}
  \label{super com}
The ring object $\fR$ is automatically commutative (in the super-sense) as an object of an
appropriate {\em derived} category. This follows from the fact
that in addition to being an associative ring object it is also factorizable (as a ring object).
Note that if we work with dg-categories (as opposed to usual categories) then commutativity does not
follow from factorization; but at this point we ignore all dg-subtleties. 
\end{rem}

\section{Coulomb branches of noncotangent type}

\subsection{Anomaly cancellation}
\label{ano can}
A symplectic representation $\bM$ of a reductive group $\fG$, i.e.\ a homomorphism
$\fG\to\Sp(\bM)=G$ gives rise to a morphism $s\colon\Gr_\fG\to\Gr_G$. The pullback $s^*\CalD$ of
the determinant line bundle of $\Gr_G$ is a factorization line bundle $\CL$ on $\Gr_\fG$
(i.e.\ it extends to the Beilinson-Drinfeld Grassmannian over the affine line $\BA^1$).
Recall from~\S\ref{generalities} that the group $\on{Pic}_{\on{fact},\BA^1}(\Gr_\fG)$ of isomorphism
classes of factorization line bundles
is canonically isomorphic to $H^2_{\on{Zar}}(\on{B}\!\fG,\ul{K}{}_2)\cong H^4(\on{B}\!\fG,\BZ(2))$
(motivic cohomology of the classifying space of $\fG$), that in turn is canonically isomorphic
to the group $\on{Quad}(X_*(\fT))^W$ of {\em even-valued} quadratic forms on the coweight lattice
$X_*(\fT)$ of a Cartan torus $\fT\subset\fG$, invariant with respect to the Weyl group $W(\fG,\fT)$.
One can also consider the group $\on{SPic}_{\on{fact},\BA^1}(\Gr_\fG)$
of isomorphism classes of factorization {\em super} line bundles. It is canonically isomorphic
to the group $\on{Bil}(X_*(\fT))^W$ of integer-valued $W$-invariant symmetric bilinear forms $B$ on
$X_*(\fT)$ such that $B(\lambda,\lambda)\in2\BZ$ for any $\lambda$ in the coroot lattice
$Q\subset X_*(\fT)$. The parity of such a line bundle on a connected component of $\Gr_\fG$ is equal
to the parity of $B(\mu,\mu)$ for any coweight $\mu$ representing a $\fT$-fixed point in this
component.

In particular, the bilinear form $B$ corresponding to $\CL$ is nothing but the pullback of the
trace form $\on{Tr}$ on $\fg=\fsp(\bM)$ (it assumes all even values, and $\CL$ is purely even).
In case $B/2\in\on{Bil}(X_*(\fT))^W$ , there exists a factorization super line bundle $\sqrt\CL$
(defined up to a non-unique isomorphism: we have $\on{Aut}(\sqrt\CL)=\Hom(\pi_1(\fG),\{\pm1\})$).
We choose such a square root, and the pullback of the gerbe $\tGr_G$ trivializes.
Hence the pullback $\CA_{\fG,\bM,\sqrt\CL}:=s^!\RH\fR$ can be viewed as a
ring object of $D_{\fG_\CO}(\Gr_\fG)$ (no twisting).

\begin{rem}
  \label{integra}
According to~Remark~\ref{integral}, $\CA_{\fG,\bM,\sqrt\CL}$ carries a canonical integral structure,
and any costalk of $\CA_{\fG,\bM,\sqrt\CL}$ is equipped with a signed basis.
\end{rem}

\begin{prop}
  \label{theo}
  The bilinear form $B$ is divisible by 2 (and $B/2$ assumes even values on all the coroots)
  iff the induced morphism $\pi_4\fG\to\pi_4G=\BZ/2\BZ$ is trivial.
\end{prop}

For a proof, see~Appendix~\ref{t vs a}.

\begin{rem}
  The second condition of the proposition is the anomaly cancellation condition of~\cite{w}.
\end{rem}

In case the anomaly cancellation condition holds true, we can consider the ring
$\CA(\fG,\bM,\sqrt\CL):=H^\bullet_{\fG_\CO}(\Gr_\fG,\CA_{\fG,\bM,\sqrt\CL})$.
Since the universal ring object $\RH\fR$ is commutative by~Remark~\ref{super com}, the ring object
  $\CA_{\fG,\bM,\sqrt\CL}$ is commutative as well.
Recall that $\sqrt\CL$ is a {\em super} line bundle, so that the ring $\CA(\fG,\bM,\sqrt\CL)$
carries an extra $\BZ/2\BZ$-grading (according to the $\BZ/2\BZ$-grading of $\pi_0(\Gr_\fG)$) in
addition to its $\BZ$-grading by cohomological degree. We conclude that $\CA(\fG,\bM,\sqrt\CL)$
is also commutative as a graded super-ring, i.e.\
{\em super-commutative} with respect to the total $\BZ/2\BZ$-grading equal to the sum of the above
$\BZ/2\BZ$-grading and the parity of the $\BZ$-grading.
Now recall that the trace form $\on{Tr}$ on the coweight lattice $X_*(T)$ of the diagonal Cartan
torus $T\subset\Sp(\bM)$ assumes all even values, and for
$\lambda=(\lambda_1,\ldots,\lambda_n)\in X_*(T)$, the value
$\frac12\on{Tr}(\lambda,\lambda)=\lambda_1^2+\ldots+\lambda_n^2$
is odd iff $\lambda$ is odd, i.e.\ $\lambda_1+\ldots+\lambda_n$ is odd. Hence by~Lemma~\ref{costalks},
the ring $\CA(\fG,\bM,\sqrt\CL)$ is entirely even with respect to the total degree,
hence the ring $\CA(\fG,\bM,\sqrt\CL)$ is simply commutative,
and the Coulomb branch $\CM_C(\fG,\bM,\sqrt\CL)$ is defined as $\on{Spec}\CA(\fG,\bM,\sqrt\CL)$.

\begin{rem}
Let $\fT\subset\fG$ be a Cartan torus. Just as in~\cite[Lemma 5.3]{bfna}, the
$H^\bullet_\fT(\pt)$-module $H^\bullet_{\fT_\CO}(\Gr_\fG,\CA_{\fG,\bM,\sqrt\CL})$ is flat, and the
$H^\bullet_\fG(\pt)$-module $H^\bullet_{\fG_\CO}(\Gr_\fG,\CA_{\fG,\bM,\sqrt\CL})$ is flat. Moreover, the natural
map $H^\bullet_{\fT_\CO}(\pt)\otimes_{H^\bullet_{\fG_\CO}(\pt)}H^\bullet_{\fG_\CO}(\Gr_\fG,\CA_{\fG,\bM,\sqrt\CL})\to
H^\bullet_{\fT_\CO}(\Gr_\fG,\CA_{\fG,\bM,\sqrt\CL})$ is an isomorphism, and
$H^\bullet_{\fG_\CO}(\Gr_\fG,\CA_{\fG,\bM,\sqrt\CL})=(H^\bullet_{\fT_\CO}(\Gr_\fG,\CA_{\fG,\bM,\sqrt\CL}))^W$.
Indeed, according to~Lemma~\ref{costalks}, the corestriction of the ring object
$\CA_{\fG,\bM,\sqrt\CL}$ to any $\fG_\CO$-orbit in $\GR_\fG$ is a constant sheaf with a certain
cohomological shift. The cohomology with coefficients in such a corestriction clearly satisfies the
desired properties, and the cohomology of the whole affine Grassmannian is filtered by the support,
with associated graded isomorphic to the direct sum of cohomology of the corestrictions.
\end{rem}

\begin{rem}
  \label{long}
Furthermore, just as in~\cite[Lemma 5.5]{bfna}, we have a natural bimodule structure
\[H^\bullet_{\fT_\CO}(\Gr_\fT,\iota^!\CA_{\fG,\bM,\sqrt\CL})\otimes  H^\bullet_{\fT_\CO}(\Gr_\fG,\CA_{\fG,\bM,\sqrt\CL})
\otimes H^\bullet_{\fG_\CO}(\Gr_\fG,\CA_{\fG,\bM,\sqrt\CL})\to H^\bullet_{\fT_\CO}(\Gr_\fG,\CA_{\fG,\bM,\sqrt\CL}),\]
where $\iota\colon\Gr_\fT\hookrightarrow\Gr_\fG$ is the natural closed embedding of affine
Grassmannians. To construct the left action note that we have the adjunction morphism
\[\iota_{\fT\fG!}(\iota^!\CA_{\fG,\bM,\sqrt\CL}\tilde\boxtimes\CA_{\fG,\bM,\sqrt\CL})\to
\CA_{\fG,\bM,\sqrt\CL}\tilde\boxtimes\CA_{\fG,\bM,\sqrt\CL}\] on the convolution diagram
$\Gr_\fG\tilde\times\Gr_\fG$ where $\iota_{\fT\fG}$ stands for the closed embedding of convolution
diagrams $\Gr_\fT\tilde\times\Gr_\fG\hookrightarrow\Gr_\fG\tilde\times\Gr_\fG$. Now apply the
functor $m_*$ (for the convolution morphism $m\colon\Gr_\fG\tilde\times\Gr_\fG\to\Gr_\fG$) to
the above morphism and compose with the multiplication
$\CA_{\fG,\bM,\sqrt\CL}\star\CA_{\fG,\bM,\sqrt\CL}\to\CA_{\fG,\bM,\sqrt\CL}$ to obtain the module structure
$\iota^!\CA_{\fG,\bM,\sqrt\CL}\star\Res_{\fG_\CO}^{\fT_\CO}\CA_{\fG,\bM,\sqrt\CL}\to\Res_{\fG_\CO}^{\fT_\CO}\CA_{\fG,\bM,\sqrt\CL}$
with respect to the natural action of $D_\fT(\Gr_\fT)$ on $D_\fT(\Gr_\fG)$.

Now the same argument as in~\cite[Lemma 5.10]{bfna} proves that
$\iota_!\colon (H^\bullet_{\fT_\CO}(\Gr_\fT,\iota^!\CA_{\fG,\bM,\sqrt\CL}))^W\to
H^\bullet_{\fG_\CO}(\Gr_\fG,\CA_{\fG,\bM,\sqrt\CL})$ is an algebra homomorphism. And
the same argument as in~\cite[Lemma 5.17]{bfna} proves that 
$\iota_!\colon H^\bullet_{\fT_\CO}(\Gr_\fT,\iota^!\CA_{\fG,\bM,\sqrt\CL})\to
H^\bullet_{\fT_\CO}(\Gr_\fG,\CA_{\fG,\bM,\sqrt\CL})\cong H^\bullet_{\fG_\CO}(\Gr_\fG,\CA_{\fG,\bM,\sqrt\CL})
\otimes_{H^\bullet_\fG(\pt)}H^\bullet_\fT(\pt)$ is an algebra homomorphism as well. Here the RHS is
equipped with the (commutative) algebra structure by extending scalars from $H^\bullet_\fG(\pt)$
to $H^\bullet_\fT(\pt)$.
\end{rem}

\subsection{Cotangent type}
Assume that a symplectic representation $\bM$ of a reductive group $\fG$ splits as $\bM=\bN\oplus\bN^*$
for some $\fG$-module $\bN$. Then the anomaly cancellation condition holds true, and moreover we
have a canonical choice $\sqrt\CL_\bN$ of the (super) square root of $\CL$. So we obtain a
ring object $\CA_{\fG,\bM,\sqrt\CL_\bN}\in D_{\fG_\CO}(\Gr_\fG)$. On the other hand, a ring object
$\CA_{\fG,\bN}:=\pi_*\bomega_\CR[-2\dim\bN_\CO]\in D_{\fG_\CO}(\Gr_\fG)$ is defined in~\cite[2(ii)]{bfnc},
such that $\CA(\fG,\bN)=H^\bullet_{\fG_\CO}(\Gr_\fG,\CA_{\fG,\bN})$ (the ring of functions on the
Coulomb branch of cotangent type).

We introduce the following cohomological degree renormalization of $\CA_{\fG,\bN}$ as
in~\cite[Remark 2.8.(2)]{bfna}. We consider a weight $\chi^\svee_\bN$ of $\fG$ equal to the sum
of all weights of $\bN$ taken with their multiplicities in $\bN$ (e.g.\ if $\fG$ is semisimple, then
$\chi^\svee_\bN=0$ for any $\bN$). For any connected component $\varpi\in\pi_0(\Gr_\fG)=\pi_1(\fG)$
and a coweight $\lambda$ of $\fG$ such that the $\fG_\CO$-orbit $\Gr^\lambda_\fG$ lies in $\varpi$,
the pairing $\langle\chi^\svee_\bN,\lambda\rangle$ depends only on $\varpi$ and is denoted
$\langle\chi^\svee_\bN,\varpi\rangle$. On the connected component $\varpi$, we consider the
cohomological shift $\CA_{\fG,\bN}|_\varpi[-\langle\chi^\svee_\bN,\varpi\rangle]$. Finally, we sum up
over all the connected components: $\CA'_{\fG,\bN}:=\bigoplus_{\varpi\in\pi_0(\Gr_\fG)}
\CA_{\fG,\bN}|_\varpi[-\langle\chi^\svee_\bN,\varpi\rangle]$.\footnote{We thank the referee for correcting
our mistake in the shift in an earlier version.}

\begin{lem}
  We have an isomorphism of ring objects $\CA'_{\fG,\bN}\cong\CA_{\fG,\bM,\sqrt\CL_\bN}$.
\end{lem}

\begin{proof}
The monoidal category $\rm{D}\modu(\Gr_\fG)^{\fG_\CO}$ acts on
$(\CW\modu)^{\fG_\CO}\cong\rm{D}\modu(\bN_\CK)^{\fG_\CO}$,
and $\CA^{DR}_{\fG,\bM,\sqrt\CL_\bN}:=\CHom(\delta_{\bN_\CO},\delta_{\bN_\CO})$. By definition, it represents the functor
$\rm{D}\modu(\Gr_\fG)^{\fG_\CO}\ni\CG\mapsto
\Hom_{\rm{D}\modu(\bN_\CK)^{\fG_\CO}}(\CG\star\delta_{\bN_\CO},\delta_{\bN_\CO})$.
Now $\CA_{\fG,\bM,\sqrt\CL_\bN}\in D_{\fG_\CO}(\Gr_\fG)$ is the image of $\CA^{DR}_{\fG,\bM,\sqrt\CL_\bN}\in\rm{D}\modu(\Gr_\fG)^{\fG_\CO}$
under the Riemann--Hilbert correspondence.

More generally, given a group $H$ acting on a variety $X$ we denote by
\[H\xleftarrow{\pr_H}H\times X\xrightarrow{a,\pr_X}X\] the natural projections and the action morphism.
The monoidal derived constructible category $D(H)$ (with respect to convolution) acts on $D(X)$
(by convolution), and given $\CF\in D(X)$, the internal Hom object $\CHom(\CF,\CF)\in D(H)$ is given
explicitly by $\CHom(\CF,\CF)=\pr_{H*}\ul\Hom(\pr_X^!\CF,a^!\CF)$, where
$\ul\Hom(\CX,\CY)=\BD\CX\otimes^!\CY$.

It is a ring object in $D(H)$ with respect to the convolution $\CA\star\CB:=m_!(\CA\boxtimes\CB)$,
where $m\colon H\times H\to H$ is the multiplication. The ring structure can be described explicitly as
follows. Since $m^!$ is the right adjoint of $m_!$, it suffices to construct a ``multiplication''
morphism $\CHom(\CF,\CF)\boxtimes\CHom(\CF,\CF)\to m^!\CHom(\CF,\CF)$, i.e.\ from
$\pr_{H*}\ul\Hom(\pr_X^!\CF,a^!\CF)\boxtimes\pr_{H*}\ul\Hom(\pr_X^!\CF,a^!\CF)$ to
\[m^!\pr_{H*}\ul\Hom(\pr_X^!\CF,a^!\CF)=\pr_{H\times H*}m^!\Hom(\pr_X^!\CF,a^!\CF)=
\pr_{H\times H*}(\pr_X^*\BD\CF\otimes^!am^!\CF),\]
where \(H\times H\xleftarrow{\pr_{H\times H}}H\times H\times X\xrightarrow{am,\pr_X}X.\)
The desired morphism is $\pr_{H\times H*}$ applied to the composition
\begin{multline*}
  \ul\Hom(\pr_X^!\CF,a^!\CF)\boxtimes\ul\Hom(\pr_X^!\CF,a^!\CF)\to
  \Delta_X^*\left(\ul\Hom(\pr_X^!\CF,a^!\CF)\boxtimes\ul\Hom(\pr_X^!\CF,a^!\CF)\right)\\
  \xrightarrow{\phi}\pr_X^*\BD\CF\otimes^!am^!\CF,
  \end{multline*}
where $\phi$ is constructed as follows. We consider $H\times H\times X\times X\times X\times X$
with projections $\pr_{X1},\pr_{X2},\pr_{X3},\pr_{X4}$ to $X$, and the closed embeddings
\[a_{11},a_{23},\Delta_{23}\colon H\times H\times X\times X\times X\hookrightarrow
H\times H\times X\times X\times X\times X,\]
\[a_{11}(h_1,h_2,x_1,x_3,x_4)=(h_1,h_2,x_1,h_1x_1,x_3,x_4),\]
\[a_{23}(h_1,h_2,x_1,x_2,x_3)=(h_1,h_2,x_1,x_2,x_3,h_2x_3),\]
\[\Delta_{23}(h_1,h_2,x_1,x_2,x_4)=(h_1,h_2,x_1,x_2,x_2,x_4).\]
Thus we identify $H^2\times X^3$ with a closed subvariety of $H^2\times X^4$ in~3 different ways.
In particular, in the next formula, $\Delta_{23}^*a_{11}^!a_{23}^!(\ldots)$ denotes a sheaf on
$H^2\times X^4$ supported on the intersection of the above~3 closed subvarieties.
Then \[\Delta_X^*\left(\ul\Hom(\pr_X^!\CF,a^!\CF)\boxtimes\ul\Hom(\pr_X^!\CF,a^!\CF)\right)=
\Delta_{23}^*a_{11}^!a_{23}^!(\bomega_H\boxtimes\bomega_H\boxtimes\BD\CF\boxtimes\CF\boxtimes\BD\CF
\boxtimes\CF)\]
\[=a_{11}^!a_{23}^!\Delta_{23}^*(\bomega_H\boxtimes\bomega_H\boxtimes\BD\CF\boxtimes\CF\boxtimes\BD\CF
\boxtimes\CF)\to a_{11}^!a_{23}^!(\BD\CF\boxtimes\Delta_{X*}\bomega_{H\times H\times X}\boxtimes\CF),\]
where $\Delta_X\colon H\times H\times X\hookrightarrow H\times H\times X\times X,\ (h_1,h_2,x_2)\mapsto
(h_1,h_2,x_2,x_2)$, and the lower morphism arises from the canonical ``evaluation'' morphism
$\CF\boxtimes\BD\CF\to\Delta_*\bomega_X$. Finally,
$a_{11}^!a_{23}^!(\BD\CF\boxtimes\Delta_{X*}\bomega_{H\times H\times X}\boxtimes\CF)$ is nothing but
$\pr_X^*\BD\CF\otimes^!am^!\CF$ (the support of the former sheaf is $H\times H\times X$ embedded by
$(h_1,h_2,x)\mapsto(h_1,h_2,x_1,h_1x_1,h_1x_1,h_2h_1x_1)$ into $H^2\times X^4$).

\medskip

Now let $Y\subset X$ be a smooth subvariety, and $\CF=\ul\BC_Y$.
Set \[Z:=\{(h,y)\in H\times Y\ :\ hy\in Y\}\subset H\times X.\]
Then $\CHom(\CF,\CF)=\pr_{H*}\bomega_Z[-2\dim Y]$. 

Similar statement applies to the situation when $H$ comes with a closed subgroup $A$ such that $Y$ is
$A$-invariant, and we consider the action of $D(A\backslash H/A)$ on $D(X)^A$.

Applying this to $H=\fG_\CK,A=\fG_\CO,\ X=\bN_\CK,\ Y=\bN_\CO$ we obtain the desired isomorphism
$\CA_{\fG,\bM,\sqrt\CL_\bN}\cong\CA'_{\fG,\bN}:=\bigoplus_{\varpi\in\pi_0(\Gr_\fG)}
\pi_*\bomega_\CR[-2\dim\bN_\CO-\langle\chi^\svee_\bN,\varpi\rangle]\in D_{\fG_\CO}(\Gr_\fG)$ up to a
cohomological shift (see~\cite[2(ii)]{bfna} for the meaning of the cohomological shift
$\bomega_\CR[-2\dim\bN_\CO]$). To determine the shift in
question it suffices to restrict to a Cartan torus in $\fG$. Then the problem reduces to the case
of a 1-dimensional representation (a character) $\bN$ of a torus. Now this character factors through
the basic character of a 1-dimensional torus, and thus the problem is reduced to the case when
$\fG=\BG_m$, and $\bN$ is the identity character. Then $\Sp(\bM)=\SL(2)$, and the corestriction
of $\RH\fR$ to the $2i$-dimensional $\SL(2,\CO)$-orbit in $\Gr_{\SL(2)}$ lives in cohomological
degree $-3i$, see~Lemma~\ref{costalks}. Hence its corestriction to the torus-fixed points in
this orbit lives in cohomological degree $i$. This completes the proof.
\end{proof}

\subsection{The abelian case}
In case $\fG$ is a torus $\fT$, we can be a little bit more explicit.

\subsubsection{} We have the groupoid of
square roots of $\CL$ on $\Gr_\fT$. The group of automorphisms $\on{Aut}(\sqrt{\CL})$ is canonically
isomorphic to the finite abelian group $\Hom(X_*(\fT),\{\pm1\})$ for any $\sqrt{\CL}$.
A choice of polarization $\bM=\bN\oplus\bN^*$ gives rise to an object $\sqrt{\CL}_\bN$ of our groupoid.
Namely, $\CL$ is the tensor product of the corresponding determinant line bundles for 2-dimensional
symplectic summands of $\bM$, and it suffices to consider the case of 2-dimensional $\bM$ and
1-dimensional $\bN$ corresponding to a character $\xi$ of $\fT$. Then
at any point $\lambda\in X_*(\fT)=\Gr_\fT$, the fiber of $\CL$ is the tensor product of two factors
depending on the sign of $\xi(\lambda)$:
\[\CL_\lambda=\begin{cases}
\det^{-1}(t^{\xi(\lambda)}\CO/\CO)\otimes\det(\CO/t^{-\xi(\lambda)}\CO)\ \on{if}\ \xi(\lambda)\leq0,\\
\det(\CO/t^{\xi(\lambda)}\CO)\otimes\det^{-1}(t^{-\xi(\lambda)}\CO/\CO)\ \on{if}\ \xi(\lambda)\geq0.
\end{cases}\]
We have a perfect pairing between $t^{\xi(\lambda)}\CO/\CO$ and $\CO/t^{-\xi(\lambda)}\CO$ when
$\xi(\lambda)\leq0$ (resp.\ between $\CO/t^{\xi(\lambda)}\CO$ and $t^{-\xi(\lambda)}\CO/\CO$
when $\xi(\lambda)\geq0$): $(g,f)=\Res(gfdt)$. We obtain the identification of
the two factors of $\CL_\lambda$,
\[\det{}\!^{-1}(t^{\xi(\lambda)}\CO/\CO)\simeq\det(\CO/t^{-\xi(\lambda)}\CO)\ \on{if}\ \xi(\lambda)\leq0,\]
\[\det(\CO/t^{\xi(\lambda)}\CO)\simeq\det{}\!^{-1}(t^{-\xi(\lambda)}\CO/\CO)\ \on{if}\ \xi(\lambda)\geq0,\]
and the first factor is the $\lambda$-fiber $(\sqrt{\CL}_\bN)_\lambda$
of the desired square root $\sqrt{\CL}_\bN$. Thus we obtain the groupoid $\on{Pol}_\bM$ whose objects
are polarizations $\bM=\bN\oplus\bN^*$ (and morphisms are the isomorphisms between the corresponding
square roots of $\CL$).

\subsubsection{}
Recall the setup of~\cite[\S4(v)]{bfna}. Given a polarization $\bN$ and another one $\bN_i$ differing
by a one-dimensional summand $\xi_i\leadsto-\xi_i$, we produce a morphism $\bsigma_{\bN_i}^\bN$ from
$\sqrt{\CL}_{\bN_i}$ to $\sqrt{\CL}_\bN$ in our groupoid. Namely, we choose a base vector $e_i$ in the
relevant summand of $\bN$ and the dual base vector $e_i^*$ in the relevant summand of $\bN_i$, so that
$\langle e_i,e_i^*\rangle=1$. The relevant factors of fibers
of $\sqrt{\CL}_{\bN_i}$ at $\lambda$ are
\[\det(e_i^*\otimes\CO/t^{-\xi_i(\lambda)}\CO)\ \on{if}\ \xi_i(\lambda)\leq0,\ \on{and}\ 
\det{}\!^{-1}(e_i^*\otimes t^{-\xi_i(\lambda)}\CO/\CO)\ \on{if}\ \xi_i(\lambda)\geq0.\]
Equivalently, we choose $N\gg0$ and consider the finite dimensional subquotient
$e_i^*\otimes t^{-2N}\CO/t^{2N}\CO$. Then the relevant factor of $(\sqrt{\CL}_{\bN_i})_\lambda$ is
\[\det{}\!^{-1}(e_i^*\otimes t^{-\xi_i(\lambda)}\CO/t^{2N}\CO)\otimes\det(e_i^*\otimes\CO/t^{2N}\CO)\]
(and we have a canonical identification with the similar expression for $N\leadsto N+1$, whence the
independence of the choice of $N\gg0$). Similarly, the relevant factor of $(\sqrt{\CL}_\bN)_\lambda$
is $\det^{-1}(e_i\otimes t^{\xi_i(\lambda)}\CO/t^{2N}\CO)\otimes\det(e_i\otimes\CO/t^{2N}\CO)$.

Recall that $t^{-2N}\CO/t^{2N}\CO$ is equipped with a nondegenerate symmetric bilinear form
$(g,f)=\Res(gfdt)$. Hence $(\BC e_i\oplus\BC e_i^*)\otimes t^{-2N}\CO/t^{2N}\CO$ is equipped with
a symplectic form $\langle,\rangle$.
It extends to the same named nondegenerate bilinear form on the exterior
algebra $\Lambda^\bullet((\BC e_i\oplus\BC e_i^*)\otimes t^{-2N}\CO/t^{2N}\CO)$. We choose an orientation,
that is an element
$\omega^*\wedge\omega=(e_i^*t^{-2N}\wedge\ldots\wedge e_i^*t^{2N-1})\wedge
(e_it^{-2N}\wedge\ldots\wedge e_it^{2N-1})$. Then we get the Hodge star operator
\[\star\colon\Lambda^\bullet((\BC e_i\oplus\BC e_i^*)\otimes t^{-2N}\CO/t^{2N}\CO)\to
\Lambda^{8N-\bullet}((\BC e_i\oplus\BC e_i^*)\otimes t^{-2N}\CO/t^{2N}\CO)\] characterized by the
requirement $z\wedge\star y=\langle z,y\rangle\omega^*\wedge\omega$.
We also have an antiautomorphism $a$ of $\Lambda^\bullet(e_i\otimes t^{-2N}\CO/t^{2N}\CO)$ identical
on the generators.

We define the super Fourier transform $\on{SFT}\colon\Lambda^\bullet(e_i\otimes t^{-2N}\CO/t^{2N}\CO)
\to\Lambda^{4N-\bullet}(e_i^*\otimes t^{-2N}\CO/t^{2N}\CO)$ by the requirement
$\on{SFT}(y)\wedge\omega=\star(ay)$. It intertwines the left (resp.\ right) multiplication by
generators with the left (resp.\ negative right) contraction with generators.
Similarly we define $\on{SFT}\colon\Lambda^\bullet(e_i^*\otimes t^{-2N}\CO/t^{2N}\CO)
\to\Lambda^{4N-\bullet}(e_i\otimes t^{-2N}\CO/t^{2N}\CO)$ by the requirement
$\on{SFT}(x)\wedge\omega^*=\star(ax)$. We have $\on{SFT}^2(y)=(-1)^{\deg y}y$.

\begin{lem}
$\on{SFT}(e_i^*t^k\wedge e_i^*t^{k+1}\wedge\ldots\wedge e_i^*t^{2N-1})
  =e_it^{-k}\wedge e_it^{-k+1}\wedge\ldots\wedge e_it^{2N-1}$, and
  $\on{SFT}(e_it^k\wedge e_it^{k+1}\wedge\ldots\wedge e_it^{2N-1})
  =(-1)^ke_i^*t^{-k}\wedge e_i^*t^{-k+1}\wedge\ldots\wedge e_i^*t^{2N-1}$.
  \hfill $\Box$
  \end{lem}

The Grassmannian $\Gr(m,e_i^*\otimes t^{-2N}\CO/t^{2N}\CO)$ of $m$-dimensional subspaces carries the
tautological vector bundle $\CS$ and the determinant line bundle
$\CL=\det^{-1}\CS\otimes\det(e_i^*\otimes \CO/t^{2N}\CO)$. The isomorphism
\[\Gr(m,e_i^*\otimes t^{-2N}\CO/t^{2N}\CO)\iso\Gr(4N-m,e_i\otimes t^{-2N}\CO/t^{2N}\CO),\ V\mapsto V^\perp,\]
is lifted to the isomorphism of determinant line bundles
\[(\Gr(m,e_i^*\otimes t^{-2N}\CO/t^{2N}\CO),\CL)\iso(\Gr(4N-m,e_i\otimes t^{-2N}\CO/t^{2N}\CO),\CL)\]
by the super Fourier transform (tensored with the isomorphism
$\det(e_i^*\otimes\CO/t^{2N}\CO)\iso\det(e_i\otimes\CO/t^{2N}\CO)$).

Finally, the relevant factor of the desired isomorphism $(\bsigma_{\bN_i}^\bN)_\lambda$ (the other factors
being identity automorphisms): \begin{multline*}(\sqrt{\CL}_{\bN_i})_\lambda
=\det{}\!^{-1}(e_i^*\otimes t^{-\xi_i(\lambda)}\CO/t^{2N}\CO)\otimes\det(e_i^*\otimes \CO/t^{2N}\CO)\\ \iso
\det{}\!^{-1}(e_i\otimes t^{\xi_i(\lambda)}\CO/t^{2N}\CO)\otimes\det(e_i\otimes \CO/t^{2N}\CO)=
(\sqrt{\CL}_\bN)_\lambda\end{multline*}
is the super Fourier transform (tensored with the isomorphism
$\det(e_i^*\otimes\CO/t^{2N}\CO)\iso\det(e_i\otimes\CO/t^{2N}\CO)$).

The composition $\bsigma_{\bN_i}^\bN\circ\bsigma^{\bN_i}_\bN$ is the automorphism of $\sqrt{\CL}_\bN$
acting by $(-1)^{\xi_i(\lambda)}$ in the fiber at $\lambda$.

\begin{cor}
  \label{composition}
The composition of isomorphisms
\[\CA(\fT,\bN_i)\cong\CA(\fT,\bM,\sqrt{\CL}_{\bN_i})\xrightarrow{\bsigma_{\bN_i}^\bN}
\CA(\fT,\bM,\sqrt{\CL}_\bN)\cong\CA(\fT,\bN)\] is the isomorphism $\sigma$ of~\cite[\S4(v)]{bfna}. 
\end{cor}

\begin{proof}
Let $\xi_i(\lambda)\leq0$. Then $r^\lambda_{\bN_i}\in\CA(\fT,\bN_i)$ is the fundamental class of the
fiber $\CR(\fT,\bN_i)_\lambda$, i.e.\ the fundamental class of the linear subspace
$V=e^*\otimes(\BC t^{-\xi_i(\lambda)}\oplus\cdots\oplus\BC t^{2N-1})\subset e^*\otimes t^{-2N}\CO/t^{2N}\CO$.
It is the `positive' generator of
$\Hom^{-2\xi_i(\lambda)}(\underline\BC{}_V,\underline\BC{}_{e^*\otimes\CO/t^{2N}\CO})$.
And it goes to the `positive' generator of
$\Hom^0(\underline\BC{}_{V^\perp},\underline\BC{}_{e\otimes\CO/t^{2N}\CO})$, that is the fundamental
class $r^\lambda_\bN\in\CA(\fT,\bN)$ of the fiber $\CR(\fT,\bN)_\lambda=e\otimes\CO/t^{2N}\CO$.

The remaining classes $r^\lambda_{\bN_i},\ \xi_i(\lambda)>0$, must go to $(-1)^{\xi_i(\lambda)}r^\lambda_\bN$
just because the composition in question is an algebra isoomorphism.
\end{proof}

\begin{rem}
The isomorphism of~Corollary~\ref{composition} is induced by the Fourier-Sato transform
$\on{FST}\colon D_{\on{constr}}^{\on{mon}}(e^*\otimes t^{-2N}\CO/t^{2N}\CO)\to
D_{\on{constr}}^{\on{mon}}(e\otimes t^{-2N}\CO/t^{2N}\CO)$ taking $\underline\BC{}_V[\dim V]$ to
$\underline\BC{}_{V^\perp}[\dim V^\perp]$. Namely, it is the action of $\on{FST}$ on the $\Hom$-spaces
between such sheaves.
\end{rem}

\subsection{Finite generation}
\begin{lem}
  $\CA(\fG,\bM,\sqrt\CL)$ is a finitely generated integral domain.
\end{lem}

\begin{proof}
We essentially repeat the argument of~\cite[6(iii)]{bfna}.
We choose a Cartan torus $\fT\subset\fG$, restrict our symplectic representation $\bM$ from $\fG$
to $\fT$, and consider the corresponding ring $\CA(\fT,\bM,\sqrt\CL)$. Note that the $\fT$-module $\bM$ is
automatically of cotangent type, i.e.\ $\bM\simeq\bN\oplus\bN^*$ for a $\fT$-module $\bN$.
In notation of~\cite[3(iv)]{bfna}, we have
$\CA(\fT,\bM,\sqrt\CL)\simeq\CA(\fT,\bM,\sqrt\CL_\bN)=\CA(\fT,\bN)$.
Similarly to~\cite[Lemma 5.17]{bfna} (see~Remark~\ref{long}), we obtain an injective homomorphism
$\CA(\fT,\bM,\sqrt\CL)\hookrightarrow\CA(\fG,\bM,\sqrt\CL)\otimes_{H^\bullet_\fG(\pt)}H^\bullet_\fT(\pt)$.

Since $\Gr_\fG$ is the union of its spherical Schubert subvarieties, we obtain a filtration by
support on $\CA(\fG,\bM,\sqrt\CL)$ (and the induced filtration on $\CA(\fT,\bM,\sqrt\CL)$) numbered by the cone
$X^+_*(\fG)$ of dominant coweights of $\fG$. For $\lambda\in X^+_*(\fG)$ let $i_\lambda$ denote
the locally closed embedding $\Gr^\lambda_\fG\hookrightarrow\Gr_\fG$. The key observation is that
$i_\lambda^!\CA_{\fG,\bM,\sqrt\CL}$ is a trivial one-dimensional local system on $\Gr^\lambda_\fG$ (shifted to
some cohomological degree determined by the monopole formula). It gives rise to an element
$[\CR_\lambda]\in\on{gr}\CA(\fG,\bM,\sqrt\CL)$ (in the cotangent case this element was the fundamental class
of the preimage of $\Gr^\lambda_\fG$ in the variety of triples, hence the notation).

Now the proof of~\cite[Proposition 6.2, Proposition 6.8]{bfna} goes through word for word in our
situation and establishes the desired finite generation.
\end{proof}

\subsection{Normality}
\begin{lem}
  $\CA(\fG,\bM,\sqrt\CL)$ is integrally closed.
\end{lem}

\begin{proof}
Again we repeat the argument of~\cite[6(v)]{bfna} with minor modifications. It reduces to an
explicit calculation of $\CA(\fG,\bM,\sqrt\CL)$ for $\fG=\SL(2)$ or $\fG=\PGL(2)$ as in~\cite[Lemma 6.9]{bfna}.
Now any symplectic representation of $\PGL(2)$ is of cotangent type
(since any irreducible representation is
odd-dimensional), so $\CA(\PGL(2),\bM,\sqrt\CL)$ is already computed in~\cite[Lemma 6.9(2)]{bfna}.
For $\SL(2)$, a representation $\bM=\oplus_{k\in\BN}V^k\otimes M^k$ (where $V^k$ is an irreducible
$\SL(2)$-module of dimension $k+1$, and $M^k$ is a multiplicity space) is symplectic iff
$\dim M^k$ is even for $k$ even. Furthermore, it is easy to see that the anomaly cancellation
condition is that the sum $\sum_{\ell\in\BN}\dim M^{4\ell+1}$ must be even. Equivalently, if for a
weight $\chi\in X^*(\SL(2))=\BZ$ we denote by $m_\chi$ the dimension of the $\chi$-weight space of
$\bM$, then $N:=\sum_{\chi\in\BZ}|\chi|m_\chi/4$ must be integral.

Then the same argument as in the proof of~\cite[Lemma 6.9(1)]{bfna} identifies $\CA(\SL(2),\bM,\sqrt\CL)$
as an algebra with~3 generators $\delta,\xi,\eta$ and a single relation
$\xi^2=\delta\eta^2-\delta^{N-1}$ if $N>0$, and $\xi^2=\delta\eta^2+\eta$ if $N=0$.
In particular, it is always integrally closed.
\end{proof}

\section{Odds and ends}

\subsection{An orthosymplectic construction of $\fK$}
\label{ortho}
The invariants $\Sym^\bullet(\fg[-2])^G$ form a free graded commutative algebra $\BC[\Sigma^\bullet_\fg]$
with generators in degrees $4,8,\ldots,4n$ (functions on a graded version of Kostant slice).
Recall the ring object $\fK$ of $D^G(\Sym^\bullet(\fg[-2]))$ introduced in~\S\ref{computa}.
It is well known that $\fK\simeq\BC[G\times\Sigma^\bullet_\fg]$, where $G$ acts in the RHS via
$g\cdot(g',\sigma)=(gg',\sigma)$, and the morphism $G\times\Sigma^\bullet_\fg\to\fg^*[2]$ is
$(g,\sigma)\mapsto\on{Ad}_g\sigma$.

Let us present one more construction of $\fK$. We take a $2n+1$-dimensional complex vector space
$\bM'$ equipped with a nondegenerate symmetric bilinear form $(\, ,)$.
Given $A\in\Hom(\bM',\bM)$ we have the adjoint operator $A^t\in\Hom(\bM,\bM')$.
We have two moment maps
\begin{multline*}
  \bq_\fg\colon\Hom(\bM,\bM')\to\fg\cong\fg^*,\ A\mapsto AA^t;\\
  \bq_{\gvee}\colon\Hom(\bM,\bM')\to\fso(\bM')=\gvee\cong(\gvee)^*,\ A\mapsto A^tA,
\end{multline*}
(we use the Killing form to identify $\fg$ (resp.\ $\gvee$) with its dual), and the
natural action $G^\vee\times G=\SO(\bM')\times\Sp(\bM)\circlearrowright\Hom(\bM',\bM)$.
We choose a maximal unipotent subgroup $U_{G^\vee}\subset G^\vee$ and a regular character
$\psi_{\gvee}$ of its Lie algebra. The hamiltonian reduction
$\BC\big[\Hom(\bM,\bM')\big]/\!\!/\!\!/(U_{G^\vee},\psi_{\gvee})$ carries the residual
action of $G$ and comoment morphism from $\Sym(\fg)$.

Now we consider $\BC\big[\Hom(\bM,\bM')\big]$ as a dg-algebra with trivial differential and with
cohomological grading such that all the generators in $\Hom(\bM,\bM')^*$ have degree~1. We will
denote this dg-algebra by $\BC\big[\Hom(\bM,\bM')[1]\big]$.\footnote{So strictly speaking we should
consider the generators in $\Hom(\bM,\bM')^*$ as having odd parity.}
Then the comoment morphisms are the homomorphisms of dg-algebras
\[\bq_\fg^*\colon\Sym{}\!^\bullet(\fg[-2])\to\BC\big[\Hom(\bM,\bM')[1]\big]
\leftarrow\Sym{}\!^\bullet(\gvee[-2])\ :\bq_{\gvee}^*,\]
and $\BC\big[\Hom(\bM,\bM')[1]\big]/\!\!/\!\!/(U_{G^\vee},\psi_{\gvee})$ is a ring object of
$D^G(\Sym^\bullet(\fg[-2]))$.

\begin{prop}
  We have an isomorphism $\fK\simeq\BC\big[\Hom(\bM,\bM')[1]\big]/\!\!/\!\!/(U_{G^\vee},\psi_{\gvee})$.
\end{prop}

\begin{proof}
We consider a locally closed subvariety $Y\subset\bM'\times\Hom(\bM',\bM)[1]$ formed by the
pairs $(v,A)$ such that $v$ is a cyclic vector for $C:=A^tA\circlearrowright\bM'$ satisfying
the orthogonality relations $(v,C^kv)=0$ for any $k<2n$ (note that for odd $k$ this orthogonality
relation is automatically satisfied), and $(v,C^{2n}v)=1$.

Clearly, $Y$ is equipped with the action of $G^\vee\times G=\SO(\bM')\times\Sp(\bM)$ and with
a morphism $\pi\colon Y\to\fg[2]\cong\fg^*[2],\ (v,A)\mapsto AA^t$.
Hence the categorical quotient
$Y/\!\!/G^\vee$ carries the residual action of $G$ and is equipped with the residual morphism
$\ol{\pi}\colon Y/\!\!/G^\vee\to\fg^*[2]$.

One can easily construct an isomorphism
$\Hom(\bM,\bM')[1]/\!\!/\!\!/(U_{G^\vee},\psi_{\gvee})\simeq Y/\!\!/G^\vee$.
We will construct an isomorphism $\BC[Y]^{G^\vee}\simeq\fK$. More precisely, we will construct an
isomorphism $\BC[Y]^{G^\vee}\simeq\BC[G\times\Sigma_\fg^\bullet]$ with gradings disregarded, and it
will be immediate to check that it respects the gradings (along with the $G$-action and the comoment
morphism).

We consider a locally closed subvariety $X\subset\bM\times\fg$ formed by the pairs
$(u,x)$ such that $u$ is a cyclic vector of $x$ satisfying the orthogonality relations
$\langle u,x^ku\rangle=0$ for any $k<2n-1$ (note that for even $k$ this orthogonality relation is
automatically satisfied), and $\langle u,x^{2n-1}u\rangle=1$.

We have an isomorphism $\eta\colon X\iso G\times\Sigma_\fg$ defined as follows.
The second factor of $\eta(u,x)$ is the image of $x$ in $\fg/\!\!/G\cong\fg^*/\!\!/G=\Sigma_\fg$.
The first factor of $\eta(u,x)$ is the symplectic $2n\times 2n$-matrix with columns
$\sC_0,\sC_1,\ldots,\sC_{2n-1}$ defined as follows. First, we set $\sC_k=x^ku$ for $k=0,\ldots,n$.
Second, we set $\sC_{n+1}=(-1)^n(x^{n+1}u-\langle u,x^{2n+1}u\rangle x^{n-1}u)$ to make sure
$\langle \sC_{n-2},\sC_{n+1}\rangle=1$ and $\langle \sC_n,\sC_{n+1}\rangle=0$.
Third, we define $\sC_{n+2}$ as $(-1)^{n-1}x^{n+2}u$ plus an
appropriate linear combination of $x^nu$ and $x^{n-2}u$ to make sure that
$\langle\sC_{n-3},\sC_{n+2}\rangle=1$, and $\sC_{n+2}$ is orthogonal to all the other previous columns.
Then we continue to apply this `Gram-Schmidt orthogonalization process' to
$x^{n+3}u,\ldots,x^{2n-1}u$ in order to obtain the desired columns $\sC_{n+3},\ldots,\sC_{2n-1}$.

Now we consider a morphism $\xi\colon Y\to X,\ (v,A)\mapsto(u=Av,\ x=AA^t)$. It factors through
$Y\to Y/\!\!/G^\vee\stackrel{\ol{\xi}}{\longrightarrow} X$, and it follows from the first
fundamental theorem of the invariant theory for $\SO(\bM')$ that $\ol{\xi}$ is an isomorphism,
cf.~\cite[proof of Lemma 2.8.1.(a)]{bft}.
\end{proof}

\subsection{The universal ring object of cotangent type}
We choose a pair of transversal Lagrangian subspaces $\bM=\bN\oplus\bN^*$. They give rise
to a (Siegel) Levi subgroup $\fG=\GL(\bN)\subset G=\Sp(\bM)$. The corresponding embedding of the
affine Grassmannians $\Gr_\fG\hookrightarrow\Gr_G$ is denoted by $s$. The pullback $s^*\CalD$ of the
determinant line bundle of $\Gr_G$ is the square of the determinant line bundle of $\Gr_\fG$.
Hence the pullback of the gerbe $\tGr_G$ trivializes, and the pullback $\sR:=s^!\RH\fR$ can be viewed
as an object of $D_{\fG_\CO}(\Gr_\fG)$ (no twisting). It is nothing but the ring object considered
in~\cite{bfnc}: the direct image of the dualizing sheaf of the variety of triples associated to
the representation $\bN$ of $\fG$ in~\cite{bfna}.

We will compute the image of $\sR$ under the derived Satake equivalence
$\Phi\colon D_{\fG_\CO}(\Gr_\fG)\iso D^\fG(\Sym^\bullet(\fgl(\bN)[-2]))$.
To this end, similarly to~\S\ref{ortho}, we introduce another copy $\bN'$ of an
$n$-dimensional complex vector space, and consider the moment map
\begin{multline*}
  \bq\colon\Hom(\bN',\bN)\times\Hom(\bN,\bN')\to\fgl(\bN)\times\fgl(\bN')\cong
  \fgl(\bN)^*\times\fgl(\bN')^*,\\ (A,B)\mapsto(AB,BA),
\end{multline*}
(we use the trace form to identify $\fgl(\bN)$ (resp.\ $\fgl(\bN')$) with its dual), and the
natural action $\GL(\bN')\times\GL(\bN)\circlearrowright\Hom(\bN',\bN)\times\Hom(\bN,\bN')$.
We choose a maximal unipotent subgroup $U\subset\GL(\bN')$ and a regular character $\psi$ of its
Lie algebra. The hamiltonian reduction
$\BC\big[\Hom(\bN',\bN)\times\Hom(\bN,\bN')\big]/\!\!/\!\!/(U,\psi)$ carries the
residual action of $\GL(\bN)$ and comoment morphism from $\Sym(\fgl(\bN))$.

Now we consider $\BC\big[\Hom(\bN',\bN)\times\Hom(\bN,\bN')\big]$ as a dg-algebra with trivial
differential and with cohomological grading such that all the generators in
$\Hom(\bN',\bN)^*\oplus\Hom(\bN,\bN')^*$ have degree~1. We will denote this algebra by
$\BC\big[\Hom(\bN',\bN)[1]\times\Hom(\bN,\bN')[1]\big]$.\footnote{So strictly speaking we should
consider the generators in $\Hom(\bN',\bN)^*\oplus\Hom(\bN,\bN')^*$ as having odd parity.}
Then the comoment morphism is a homomorphism of dg-algebras
\[\bq^*\colon\Sym{}\!^\bullet(\fgl(\bN)[-2]\oplus\fgl(\bN')[-2])
\to\BC\big[\Hom(\bN',\bN)[1]\times\Hom(\bN,\bN')[1]\big],\]
and $\BC\big[\Hom(\bN',\bN)[1]\times\Hom(\bN,\bN')[1]\big]/\!\!/\!\!/(U,\psi)$
is a ring object of $D^\fG(\Sym^\bullet(\fgl(\bN)[-2]))$.

\begin{prop}
  We have an isomorphism
  \[\Phi\sR\simeq\BC\big[\Hom(\bN',\bN)[1]\times\Hom(\bN,\bN')[1]\big]/\!\!/\!\!/(U,\psi).\]
\end{prop}

\begin{proof}
We consider an open
subvariety $Z\subset\bN'\times\Hom(\bN',\bN)[1]\times\Hom(\bN,\bN')[1]$ formed by the triples
$(v,A,B)$ such that $v$ is a cyclic vector for $BA\circlearrowright\bN'$. It is equipped with a
morphism $\varpi\colon Z\to\fgl(\bN)[2]\cong
 \fgl(\bN)^*[2],\ (v,A,B)\mapsto AB$, and a natural action of $\GL(\bN')\times\GL(\bN)$.
 Hence the categorical quotient
$Z/\!\!/\GL(\bN')$ carries the residual action of $\GL(\bN)$ and is equipped with the residual
morphism $\ol{\varpi}\colon Z/\!\!/\GL(\bN')\to\fgl(\bN)^*[2]$.

One can easily construct an isomorphism
\[\big(\Hom(\bN',\bN)[1]\times\Hom(\bN,\bN')[1]\big)/\!\!/\!\!/(U,\psi)\simeq Z/\!\!/\GL(\bN').\]
It remains to construct an isomorphism $\Phi\sR\simeq\BC[Z]^{\GL(\bN')}$ compatible with the comoment
morphisms from $\Sym^\bullet(\fgl(\bN)[-2])$ and with the actions of $\GL(\bN)$.

  The desired isomorphism is a corollary of~\cite[Theorem 3.6.1]{bfgt}. Indeed, in notation
of~\cite[\S3.2, \S3.10]{bfgt}, we have
$\sR=u_0^*(E_0\overset!\oast\bomega_{\Gr_{\GL(\bN)\times\bN_\CK}}\overset!\oast E_0)$ by comparison of
definitions (say $E_0$ stands for the constant sheaf on $\Gr^0_{\GL(\bN)}\times\bN_\CO$,
see~\cite[\S3.9]{bfgt}, while $\bomega$ stands for the dualizing sheaf). So we have to compute this
triple convolution in terms of the mirabolic
Satake equivalence. The corresponding convolution on the coherent side is defined
in~\cite[\S\S3.4,3.5]{bfgt}. The convolution of~3 objects is computed via the double cyclic quiver
$\tilde{A}_3$ on~4 vertices, cf.~\cite[(3.4.1)]{bfgt}. The result of this computation is
nothing but $\BC[Z]^{\GL(\bN')}$.
\end{proof}

\subsection{Baby version}
\label{baby}
Let $P\subset G$ stand for the stabilizer of the Lagrangian subspace $\bN\subset\bM$
(Siegel parabolic). Let $P'\subset P$ stand for the derived subgroup.
We consider the Lagrangian Grassmannian $\LGr_\bM=G/P$. The $\mu_2$-gerbe of square roots of the ample
determinant line bundle $\CalD$ over $\LGr_\bM$ is denoted $\tLGr_\bM$. The group $P'$ acts on $\tLGr_\bM$.
We consider the derived constructible category $D^b_{P'}(\tLGr_\bM)$ of genuine sheaves on $\tLGr_\bM$
(such that $-1\in\mu_2$ acts by $-1$). An open sub-gerbe $\CT\hookrightarrow\tLGr_\bM\times\tLGr_\bM$
is formed by all the pairs of transversal Lagrangian subspaces in $\bM$.
We denote by $\tLGr_\bM\stackrel{p}{\leftarrow}\CT\stackrel{q}{\rightarrow}\tLGr_\bM$ the two projections,
and we define the Radon Transform $\RT:=p_*q^!\colon D^b_{P'}(\tLGr_\bM)\to D^b_{P'}(\tLGr_\bM)$.
Finally, we consider the $P'$-equivariant derived category $\Dmo^{P'}(\LGr_\bM)$ of $D$-modules on
$\LGr_\bM$ twisted by the negative square root of the determinant line bundle $\CalD$.
We have the Riemann--Hilbert equivalence $\RH\colon\Dmo^{P'}(\LGr_\bM)\iso
D^b_{P'}(\tLGr_\bM)$.

The Weyl algebra of the symplectic space $\bM$ is denoted by $\CW_\bM$.
The homomorphism $\fg=\fsp(\bM)\to\on{Lie}\CW_\bM$ (oscillator representation) goes back to~\cite{s},
see~\cite[\S2]{h} and~\cite[\S1.1]{la}. The restriction of the $\CW_\bM$-module $\BC[\bN]$ to
$\fg$ is a direct sum of two irreducible modules $L^{\lambda_g}\oplus L^{\lambda_s}$ (even and odd
functions). Here in the standard orthonormal basis $w_1^*,\ldots,w_n^*$ of a
Cartan Lie subalgebra of $\on{Lie}P$ we have $\lambda_g=-\frac12\sum_{i=1}^nw_i$,
and $\lambda_s=\lambda_g-w_n$.

The baby version $S$ of $\Theta$-sheaf, introduced in~\cite[Definition 2]{l} and studied
in~\cite[\S3]{ll}, is the direct sum of IC-sheaves of two $P$-orbits in $\tLGr_\bM\colon S_g$ of the
open orbit, and $S_s$ of the codimension~1 orbit. We have irreducible twisted $D$-modules
$\CS_g=\tau_{\geq0}\Loc L^{\lambda_g},\ \CS_s=\tau_{\geq0}\Loc L^{\lambda_s}$,
and $\RH(\CS_g)=S_g,\ \RH(\CS_s)=S_s$.

Finally, $\RT(S)$ is isomorphic to $S$ up to a shift. More precisely, we have
$\RT(S_g)\simeq S_s[n^2+2]$, and $\RT(S_s)\simeq S_g[n^2]$
for $n$ odd, while for $n$ even we have $\RT(S_s)\simeq S_s[n^2+2]$ and $\RT( S_g)\simeq S_g[n^2]$.
This follows e.g.\ from~\cite[Theorem 10.7]{ly}, since the Radon Transform is the convolution
with the $*$-extension of the sign local system from the open orbit in $\tLGr_\bM$.

\appendix

\section{Localization and the Radon transform}
\label{gd}

\centerline{By Gurbir Dhillon}

\bigskip

\subsection{Lie groups and algebras} Let $G$ be an almost simple, simply
connected, group and $\fg$ its  Lie algebra.\footnote{The results discussed
below straightforwardly generalize to any connected reductive group $G$.} Let
$\kappa$ be a level, i.e. an $\on{Ad}$-invariant bilinear form on $\fg$, and
consider the associated affine Lie algebra
\[
0 \rightarrow \mathbb{C} \cdot \mathbf{1} \rightarrow \gk \rightarrow
\fg(\!(t)\!) \rightarrow 0.
\]

\subsection{Levels}

Let us write $\kappa_{c}$ for the critical level, i.e., minus one half times
the Killing form. We recall that a level $\kappa$ is called {positive} if
\[
\kappa \notin  \kappa_c + \mathbb{Q}^{\geqslant 0} \cdot \kappa_c.
\]
Similarly, a level $\kappa$ is called {negative} if
\[
\kappa \notin \kappa_c - \mathbb{Q}^{\geqslant 0} \cdot \kappa_c.
\]
Note that, in this convention, an irrational multiple of the critical level is
considered both positive and negative.

\subsection{Localization on the thin Grassmannian} For any level $\kappa$, one
has a $\DMOD(G_{\CK})$-equivariant functor of global sections
\[
\Gamma_\kappa\colon  \DMOD(\Gr_G) \rightarrow \gk\mod.
\]
It is the unique equivariant functor sending the delta D-module at the trivial
coset $\delta_e$ to the vacuum module, i.e., the parabolically induced module
\[
\mathbb{V}_{\kappa} := \on{pind}_{\fg}^{\gk} \mathbb{C}.
\]

The functor admits a right adjoint. Moreover, after passing to spherical
vectors, it also admits a left adjoint. That is, one has an adjunction
\[
\on{Loc}_\kappa\colon  \gk\mod^{G_\CO} \rightleftarrows \DMOD(\Gr_G)^{G_\CO}:
\Gamma_\kappa.
\]

\subsection{Localization on the thick Grassmannian}
\label{A4}
Let us denote the usual and
dual categories of D-modules on the thick Grassmannian by
\[
\DMOD(\GR_G)_! \quad \text{and} \quad \DMOD(\GR_G)_*.
\]
By definition, if we let $\mathbf{U}_i$ range through the quasicompact open
subschemes of $\GR_G$, we have
\[
\DMOD(\GR_G)_! \simeq \underset i \varprojlim \DMOD(\mathbf{U}_i) \quad
\text{and} \quad \DMOD(\GR_G)_* \simeq \underset i \varinjlim
\DMOD(\mathbf{U}_i),
\]
where the transition maps are given by $!$-restriction and $*$-pushforward,
respectively.

Following Arkhipov--Gaitsgory \cite{ag}, one has $\DMOD(G_\CK)$-equivariant
localization and global sections functors
\begin{equation} \label{e:gamma1}
	\mathbf{Loc}_\kappa\colon  \gk\mod \rightarrow \DMOD(\GR_G)_! \quad \text{and}
	\quad \mathbf{\Gamma_\kappa}\colon  \DMOD(\GR_G)_* \rightarrow \gk\mod.
\end{equation}
%
Upon passing to spherical vectors, one has the following adjunctions, which are
sensitive to the sign of the level. If $\kappa$ is positive,
$\mathbf{Loc}_\kappa$ admits a right adjoint of (smooth) global sections
\begin{equation} \label{e:gamma2}
	\mathbf{Loc}_\kappa\colon  \gk\mod^{G_\CO} \rightleftarrows
	\DMOD(\GR_G)_!^{G_\CO}\colon \mathbf{\Gamma_\kappa}.
\end{equation}
Similarly, if $\kappa$ is negative, $\mathbf{\Gamma_\kappa}$ admits a left
adjoint
\begin{equation} \label{e:gamma3}
	\mathbf{Loc}_\kappa\colon  \gk\mod^{G_\CO} \rightleftarrows
	\DMOD(\GR_G)_*^{G_\CO}\colon \mathbf{\Gamma_\kappa}.
\end{equation}
We emphasize that the sources of the functors denoted $\mathbf{\Gamma_\kappa}$
in~\eqref{e:gamma1} and~\eqref{e:gamma2} are distinct, as are the
sources of the functors denoted $\mathbf{Loc}_\kappa$ in~\eqref{e:gamma1} and~\eqref{e:gamma3}.

\subsection{Radon Transform}
\label{A5}
For any level $\kappa$, consider the Radon transform functors
\begin{align*}
	& \on{RT}_!\colon  \DMOD(\Gr_G) \rightarrow \DMOD(\GR_G)_! \quad \text{and} \\ &
	\on{RT}_*\colon  \DMOD(\Gr_G) \rightarrow \DMOD(\GR_G)_*.
\end{align*}
These are by definition $\DMOD(G_\CK)$-equivariant, and are characterized by
sending $\delta_e$ to the $!$- and $*$-extensions of the constant intersection
cohomology D-module \[\mathbb{C}[G_\CO \cdot G_{\mathbb{C}[t^{-1}]} /
G_{\mathbb{C}[t^{-1}]}],\] respectively. In what follows, we denote these
objects by $j_!$ and $j_*$, respectively.

It is standard that $\on{RT}_!$ and $\on{RT}_*$  induce equivalences on
spherical vectors, and in particular are fully faithful embeddings.

\subsection{Global sections and the Radon transform: negative level}

We now turn to the relationship between the global sections functors on the
thin and thick Grassmannians and the Radon transform. We begin with the case of
$\kappa$ negative.

\begin{prop} \label{p:neglevel}  Suppose $\kappa$ is negative. Then the functor
of global sections on the thin Grassmannian
	\begin{equation} \label{e:option1}
		\Gamma_\kappa\colon  \DMOD(\Gr_G) \rightarrow \gk\mod
	\end{equation}
	is canonically $\DMOD(G_{\CK})$-equivariantly equivalent to the composition
	\begin{equation} \label{e:option2}
		\DMOD(\Gr_G) \xrightarrow{\on{RT}_*} \DMOD(\GR_G)_*
		\xrightarrow{\mathbf{\Gamma_\kappa}} \gk\mod.
	\end{equation}

\end{prop}

\begin{proof}It is enough to show that the composition \eqref{e:option2}  sends
$\delta_e$ to the vacuum module $\mathbb{V}_\kappa$. Unwinding definitions, we
have
	\[
	\mathbf{\Gamma_\kappa} \circ \on{RT}_* (\delta_e) \simeq
	\mathbf{\Gamma_\kappa}( j_*) \simeq \mathbb{C}[G_\CO \cdot
	G_{\mathbb{C}[t^{-1}]} / G_{\mathbb{C}[t^{-1}]}],
	\]
	i.e., $\delta_e$ is sent to the algebra of functions on the big cell. The
	function which is identically one on the cell yields, by its $G_\CO$
	invariance, a canonical map of $\gk$-modules
	\[
	\mathbb{V}_\kappa \rightarrow \mathbb{C}[G_\CO \cdot G_{\mathbb{C}[t^{-1}]}
	/ G_{\mathbb{C}[t^{-1}]}].
	\]
	It is straightforward to see that the characters of the two appearing
	modules coincide. Moreover, by our assumption on $\kappa$,
	$\mathbb{V}_\kappa$ is irreducible, hence the map is an isomorphism, as
	desired.  \end{proof}

\begin{rem} The functions on the big cell, at any level, are canonically
isomorphic to the contragredient dual of the vacuum. In particular, at a
positive rational level $\kappa$, the assertion of  Proposition
\ref{p:neglevel} is false. We will meet its corrected variant in Proposition
\ref{p:locpos} below.
\end{rem}

By taking the statement of Proposition \ref{p:neglevel}, passing to spherical
invariants, and then left adjoints, we deduce the following.

\begin{cor} \label{c:locthin} Suppose $\kappa$ is negative. Then the
localization functor on the thin Grassmannian
	\[
	\on{Loc}_\kappa\colon  \gk\mod^{G_\CO} \rightarrow \DMOD(\Gr_G)^{G_\CO}
	\]
	is canonically $\DMOD(G_\CO \backslash G_\CK / G_\CO)$-equivariantly
	equivalent to the composition
	\[
	\gk\mod^{G_\CO} \xrightarrow{\mathbf{Loc}_\kappa} \DMOD(\GR_G)_*
	\xrightarrow{ \on{RT}_*^{-1}} \DMOD(\Gr_G).
	\]
\end{cor}

\subsection{Global sections and the Radon transform: positive level}
\label{A7} Let us now turn to the case of $\kappa$ of positive level. As we
will see momentarily, the analog of the approach we took at negative level
requires knowing the global sections of a $!$-extension, and is therefore less
immediate.

\begin{prop}\label{p:locpos}Suppose $\kappa$ is positive. Then the functor of
	global sections on the thin Grassmannian
	\[
	\Gamma_\kappa\colon  \DMOD(\Gr_G) \rightarrow \gk\mod
	\]
	is canonically $\DMOD(G_\CK)$-equivariantly equivalent to the composition
	\[
	\DMOD(\Gr_G) \xrightarrow{\on{RT}_!} \DMOD(\GR_G)_!
	\xrightarrow{\mathbf{\Gamma}_\kappa} \gk\mod.
	\]

\end{prop}

\begin{proof} It is enough to show the composition sends $\delta_e$ to the
vacuum module $\mathbb{V}_\kappa$. By definition, we have that
\[
  \mathbf{\Gamma_\kappa} \circ \on{RT}_! (\delta_e) \simeq
  \mathbf{\Gamma_{\kappa}}( j_!).
\]
We will deduce the calculation of the latter global sections from the work of
Kashiwara--Tanisaki on localization at positive level \cite{kt}.

To do so, fix a Borel subgroup $B^-$ of $G$. Write
	$\mathbf{I^-}$ for the `thick Iwahori' group ind-scheme associated to
	$B^-$,
	i.e., the preimage
	of $B^-$ under the map \[G_{\CC[t^{-1}]} \rightarrow G\] given by
	evaluation
	at infinity. Write $\mathbf{Fl}_G \simeq G_\CK / \mathbf{I}^-$ for the
	thick affine flag variety. Consider the functor of (smooth) global sections
	\[
	\mathbf{\Gamma}_\kappa(\FL_G, -)\colon  \DMOD(\mathbf{Fl}_G) \rightarrow
	\gk\mod,
	\]
	which is denoted in {\em loc.cit.}\ by $\widetilde{\Gamma}$.

    Fix another Borel subgroup $B$ of $G$ in general position with $B^-$. Write
    $I$ for
	the associated Iwahori group scheme, i.e., the preimage of $B$ under
	the map
	$G_\CO \rightarrow G$ given by evaluation at zero.

	Let us denote by $\jmath_!$ the $!$-extension of the constant intersection
	cohomology D-module on
	the open orbit $I
	\cdot \mathbf{I^-} / \mathbf{I^-}$. On the other side of
	$\mathbf{\Gamma}_\kappa$,  let us denote the Verma module of
	highest weight zero for $\fg$ by $M_0$, and note the Verma
	module
	for $\gk$ of highest weight zero is given by $\on{pind}_{\fg}^{\gk} (M_0)$.

	Then, the desired result of Kashiwara--Tanisaki is the canonical
	equivalence
	\[
	    \mathbf{\Gamma}_\kappa(\FL_G,  \jmath_!) \simeq
	    \on{pind}_{\fg}^{\gk}(M_0),
	\]
	see \cite[Theorem 4.8.1(ii)]{kt}.\footnote{Strictly speaking,
	Kashiwara--Tanisaki discuss only the case of $\kappa$ positive rational,
	but their argument applies more generally to any positive $\kappa$.}

	We are ready to deduce the proposition. Consider the projection \[\pi\colon
	\FL_G \rightarrow \GR_G.\] As both functors denoted by
	$\mathbf{\Gamma}_\kappa$ are the smooth vectors in the naive global
	sections, and $\pi$ is a Zariski locally trivial fibration with fibre
	$G/B$, we have that
	\[
	    \mathbf{\Gamma}_{\kappa}(j_!) \simeq \mathbf{\Gamma}_\kappa(\FL_G,
	    \pi^{!*}(\hspace{.2mm}j_!)),
	\]
	where $\pi^{!*} := \pi^![- \dim G/B]$. If we write $\on{Av}_{!}^{I, G_\CO}$
	for
	the functor of relative $!$-averaging from $I$-invariants to
	$G_\CO$-invariants, note
	that \[\pi^{!*}(\hspace{.2mm}j_!) \simeq \on{Av}^{I,
	G_\CO}_!(\hspace{.2mm}\jmath_!).\] By the
	equivariance of the appearing functors, we then have
	\begin{multline*}
	 \mathbf{\Gamma}_\kappa(\FL_G,
	 \pi^{!*}(\hspace{.2mm}j_!)) \simeq \mathbf{\Gamma_{\kappa}}(\FL_G,
	 \on{Av}_!^{I,
	 G_\CO}(\hspace{.2mm}\jmath_!)) \\ \simeq \on{Av}_!^{I,
	 G_\CO} \circ \hspace{.2mm}\mathbf{\Gamma_{\kappa}}(\FL_G,
 \jmath_!) \simeq \on{Av}_!^{I,
 	G_\CO} \circ \on{pind}_{\fg}^{\gk}(M_0) \\ \simeq \on{pind}_\fg^{\gk} \circ
 	\on{Av}_!^{B, G} (M_0)  \simeq \on{pind}_\fg^{\gk}(\mathbb{C})  \simeq
 	\mathbb{V}_\kappa,
	\end{multline*}
as desired. \end{proof}

\begin{cor} Suppose $\kappa$ is positive. Then the functor of localization on
the thin Grassmannian
	\[
	\on{Loc}_\kappa\colon  \gk\mod^{G_\CO} \rightarrow \DMOD(\Gr_G)^{G_\CO}
	\]
	is canonically $\DMOD(G_\CO \backslash G_\CK / G_\CO)$-equivariantly
	equivalent to the composition
	\[
	\gk\mod^{G_\CO} \xrightarrow{\mathbf{Loc}_\kappa} \DMOD(\GR_G)^{G_\CO}
	\xrightarrow{\RT_!^{-1}}
	\DMOD(\Gr_G)^{G_\CO}.
	\]
\end{cor}

\begin{rem} Analogs of the results of this appendix hold, {\em mutatis
mutandis}, after replacing the thick and thin Grassmannians by any opposite
thick and
thin partial affine flag varieties, by similar arguments, as well as for
monodromic D-modules on the enhanced thick and thin affine flag varieties.
Similarly, one may replace $G_\CK$ by a quasi-split form.

With some care about hypotheses on twists, similar results hold for a
symmetrizable Kac--Moody group, again by similar arguments. We leave the
details to the interested reader.
\end{rem}

\section{Topological vs.\ algebraic anomaly cancellation condition}
\label{t vs a}
\centerline{By Theo Johnson-Freyd}

\bigskip

The goal of this appendix is to prove~Proposition~\ref{theo}.

\subsection{Simply connected case}
Let $\fG$ be a connected complex reductive group with classifying space~$\B \fG$, and let $\varrho \colon  \fG \to \Sp(2n,\BC)$ a symplectic representation of $\fG$. Recall that $H^4(\B\Sp(2n,\BC), \BZ) \cong \BZ$ is generated by the \emph{universal (quaternionic first) Pontryagin class} $q_1$. Thus $\varrho$ has a \emph{(quaternionic first) Pontryagin class} $q_1(\varrho) = \varrho^*(q_1) \in H^4(\B \fG, \BZ)$, equal (up to a sign convention) to the second Chern class of the underlying complex representation $\varrho \colon  \fG \to \Sp(2n,\BC) \to \mathrm{SL}(2n,\BC)$. Recall furthermore that $\pi_4 \Sp(2n,\BC) \cong \pi_5 \B\Sp(2n,\BC) \cong \BZ/2\BZ$.
\begin{thm}\label{thm-tjf}
  If $q_1(\varrho)$ is even, i.e.\ divisible by $2$ in $H^4(\B \fG, \BZ)$, then $\varrho$ induces the zero map $\pi_5 \varrho \colon  \pi_5 \B \fG \to \pi_5 \B\Sp(2n,\BC)$. If $\fG$ is simply connected, then the converse holds: if $\pi_5 \varrho=0$, then $q_1(\varrho)$ is even.
\end{thm}

Theorem~\ref{thm-tjf} obviously depends only on the homotopy 5-type $\tau_{\leq 5} \B\Sp(2n,\BC)$ of $\B\Sp(2n,\BC)$. This homotopy 5-type is independent of $n$, and so we will henceforth call it simply $\tau_{\leq 5} \B\Sp$. We will prove Theorem~\ref{thm-tjf} for any map $\varrho \colon  \B \fG \to \tau_{\leq 5} \B\Sp$.

\begin{rem}
  To see that simple connectivity is a necessary condition, consider
  $\varrho\colon\fG=\BC^\times\hookrightarrow\Sp(2,\BC)$ a Cartan
  torus of $\Sp(2,\BC)$. Then $q_1(\varrho)$ is a generator of $H^4(\B\fG,\BZ)$.
\end{rem}

\begin{prop}\label{prop-tjf}
  If $\fG$ is connected and simply connected, then $H_5 \B \fG$ is trivial.
\end{prop}

\begin{proof}

  Recall that $\pi_2 \fG = \pi_3 \B \fG$ vanishes and $\pi_4 \B \fG = H_4 \B \fG$ is a free abelian group.%
  \footnote{Indeed, $\pi_3 \B \fG$ vanishes for every Lie group, with no conditions, and $\pi_4 \B \fG$ is always free abelian. The Hurewicz map $\pi_4 \B \fG \to H_4 \B \fG$ is an isomorphism if $\fG$ is simply connected, in which case
  $H_4 \B \fG$ has rank equal to the number of simple factors of $\fG$.}
  Recall furthermore that $H^\bullet(\B \fG,\BQ)$ is concentrated in even degrees.%
  \footnote{$H^\bullet(\B \fG,\BQ)$ is a polynomial algebra on generators of degrees twice the exponents of $\fG$.}
  From the universal coefficient theorem, we find that $H_5 \B \fG$ is torsion.

   Choose a Borel subgroup $B \subset \fG$, and consider the flag variety $X = \fG/B$.
  The homology of $X$ is very well understood. Indeed, $X$ has a {\em Schubert decomposition} into cells of even real dimension.
  In particular, the homology of the manifold $X$ is free abelian and concentrated in even degrees.

  Consider the homological Serre spectral sequence for the fibre bundle $X \to \B B \to \B \fG$:
  \[E^2_{ij} := H_i(\B \fG, H_j X) \Rightarrow H_{i+j} \B B.\]
  The $E^2$ page vanishes whenenever $j$ is odd and also when $1\leq i \leq 3$. Since
  $B$ is homotopy equivalent to a torus, $H_\bullet \B B$ is free abelian and concentrated in even
  degrees, and hence the $E^\infty$ page vanishes when $i+j$ is odd. It follows that there is an exact
  sequence
  \[0 \to H_5 \B \fG \to H_4 X \to H_4 \B B \to H_4 \B \fG \to 0.\]
  But $H_5 \B \fG$ is torsion, whereas $H_4 X$ is free abelian.%
\end{proof}

\begin{cor}\label{cor-tjf}
  Let $\fG$ be a connected complex reductive Lie group, not necessarily simply connected,
  and let $Y$ be any topological space. Suppose given a map $\B \fG \to \tau_{\leq 4} Y$ which admits a lift to $Y$. Then any two lifts $\B \fG \to Y$ induce the same map $\pi_5 \B \fG \to \pi_5 Y$.
\end{cor}

\begin{proof}
  The lifts of a map $\B \fG \to \tau_{\leq 4} Y$ along $\tau_{\leq 5} Y \to \tau_{\leq 4} Y$, assuming there are any, form a torsor for $H^5(\B \fG, \pi_5 Y)$. Suppose two lifts differ by some class in $H^5(\B \fG, \pi_5 Y)$. Then their actions on $\pi_5 \B \fG$ differ by the image of that class along the Hurewicz map $H^5(\B \fG, \pi_5 Y) \to \Hom(\pi_5 \B \fG , \pi_5 Y)$ induced from $\pi_5 \B \fG \to H_5 \B \fG$.

 Let $\fG^{\sico}$ denote the simply connected cover of $\fG$. Then $\pi_5 \B \fG^\sico \to \pi_5 \B \fG$ is an isomorphism, and so the Hurewicz map $\pi_5 \B \fG \to H_5 \B \fG$ factors through $H_5 \B \fG^{\sico} \to H_5 \B \fG$. But $H_5 \B \fG^{\sico} = 0$ by Proposition~\ref{prop-tjf}.
\end{proof}

To complete the proof of Theorem~\ref{thm-tjf}, we will need to know the space $\tau_{\leq 5}\B\Sp$. It has precisely two nontrivial homotopy groups: $\pi_4 = \BZ$ and $\pi_5 = \BZ/2\BZ$. Thus we will know it completely if we know its Postnikov k-invariant. Recall that the Postnikov k-invariant of the extension $K(\BZ/2\BZ,5) \to \tau_{\leq 5}\B\Sp \to K(\BZ,4)$ is some universal cohomology operation $f \colon  H^4(-,\BZ) \to H^6(-,\BZ/2\BZ)$. A map $X \to K(\BZ,4)$ is, up to homotopy, a class $\alpha \in H^4(X,\BZ)$, and it lifts along $\tau_{\leq 5}\B\Sp \to K(\BZ,4)$ if and only if $f(\alpha) = 0 \in H^6(X,\BZ/2\BZ)$.

\begin{lem}
  The Postnikov $k$-invariant of the $\tau_{\leq 5} \B\Sp$ is $\Sq^2 \circ (\mathrm{mod}\,2) \colon  H^4(-,\BZ) \to H^6(-,\BZ/2\BZ)$, where $(\mathrm{mod}\,2) \colon  H^4(-,\BZ) \to H^4(-,\BZ/2\BZ)$ is the corresponding map on coefficients, and $\Sq^2$ is the second Steenrod square.
\end{lem}

\begin{proof}
  Bott periodicity identifies $\tau_{\leq 5} \B\Sp$ with the 4-fold suspension of the infinite loop space $\tau_{\leq 1} ko$.
   Thus the statement in the Lemma follows from (and is equivalent to) the fact that the k-invariant (at the level of infinite loop spaces) connecting $\pi_0 ko = \BZ$ to $\pi_1 ko = \BZ/2\BZ$ is $\Sq^2 \circ (\mathrm{mod}\,2)$.
\end{proof}

\begin{proof}[Proof of Theorem~\ref{thm-tjf}]
  Fix $\varrho \colon  \B \fG \to \tau_{\leq 5}\B\Sp$. The class $q_1(\varrho) \in H^4(\B \fG, \BZ)$ is nothing but the image of $\varrho$ along $\tau_{\leq 5}\B\Sp \to \tau_{\leq 4}\B\Sp = K(\BZ, 4)$, and note that $q_1(\varrho)$ factors through $\tau_{\leq 4} \B \fG$.

  Suppose that $q_1(\varrho)$ is even. Then $\Sq^2(q_1(\varrho) \, \mathrm{mod}\, 2) = 0$, and so $q_1(\varrho) \colon  \tau_{\leq 4} \B \fG \to \tau_{\leq 4} \B \Sp$ lifts to a map $\tau_{\leq 4} \B \fG \to \tau_{\leq 5} \B \Sp$. The composition $\B \fG \to \tau_{\leq 4} \B \fG \to \tau_{\leq 5} \B \Sp$ vanishes on $\pi_5 \B \fG$. This composition might not be equal to $\varrho$, but it and $\varrho$ are both lifts of the same map $\B \fG \to \tau_{\leq 4}\B \Sp$.
  And so by~Corollary~\ref{cor-tjf} they have the same (trivial) value on  $\pi_5 \B \fG$.

  Now suppose that $\fG$ is connected and simply connected.
  Then $\tau_{\leq 4}\B \fG \cong K(A,4)$ where $A$ is a free abelian group, and
  $H^4(\B \fG, \BZ/2\BZ) = H^4(K(A,4), \BZ/2\BZ) = \Hom(A,\BZ/2\BZ)$. We claim that
  $\Sq^2 \colon  H^4(K(A,4), \BZ/2\BZ) \to H^6(K(A,4), \BZ/2\BZ)$ is injective. Indeed, suppose that
  $\alpha \neq 0 \in  \Hom(A,\BZ/2\BZ)$, and let $a \colon  \BZ \to A$ be an element such that $\alpha(a) \neq 0$. By restricting along the corresponding map $K(\BZ,4) \to K(A,4)$, if suffices to prove the claim when $A = \BZ$ and $\alpha$ is the map that reduces mod 2. There is a nonzero map $\beta\colon  K(\BZ/2\BZ,3) \to K(\BZ,4)$, and the composition $K(\BZ/2\BZ,3) \overset\beta\to K(\BZ,4) \overset\alpha\to K(\BZ/2\BZ,4)$ is the class $\Sq^1 z \in H^4(K(\BZ/2\BZ,3), \BZ/2\BZ)$, where $z \in H^3(K(\BZ/2\BZ,3), \BZ/2\BZ)$ generates $H^3(K(\BZ/2\BZ,3), \BZ/2\BZ)$ over the Steenrod algebra. Then $\Sq^2 ( \alpha) ( \beta) = \Sq^2 \Sq^1 z \neq 0 \in H^6(K(\BZ/2\BZ,3), \BZ/2\BZ)$. It follows that $\Sq^2(\alpha) \neq 0$, proving the claim that $\Sq^2 \colon  H^4(K(A,4), \BZ/2\BZ) \to H^6(K(A,4), \BZ/2\BZ)$ is injective.

  Suppose that $\pi_5 \varrho = 0$. Then the map $\tau_{\leq 5}\varrho \colon  \tau_{\leq 5} \B \fG \to \tau_{\leq 5} \B\Sp$ factors through the cofibre of the inclusion $K(\pi_5 \B \fG,5) \to \tau_{\leq 5} \B \fG$. Note that this inclusion is the fibre of the map $\tau_{\leq 5} \B \fG \to \tau_{\leq 4} \B \fG$.
  In general, given a fibre bundle of spaces $F \to E \to B$, there is a canonical map
  $\operatorname{cofibre}(F \to E) \to B$, but it is not always an equivalence. However, assuming
  $\fG$ is connected and simply connected, then $\tau_{\leq 5}\varrho \colon  \tau_{\leq 5} \B \fG \to \tau_{\leq 5} \B\Sp$ is canonically a map of infinite loop spaces,\footnote{In general, a space all of whose homotopy groups are in degrees $(n, 2n]$ for some $n$ is automatically an infinite loop space.} and for infinite loop spaces, a fibre and cofibre sequences agree.
  In particular, if $\pi_5 \varrho = 0$ and  $\fG$ is connected and simply connected, then $\tau_{\leq 5}\varrho \colon  \tau_{\leq 5} \B \fG \to \tau_{\leq 5} \B\Sp$ factors through $\tau_{\leq 4} \B \fG$.

  But this means that $q_1(\varrho) \colon  \tau_{\leq 4} \B \fG \to \tau_{\leq 4} \B\Sp$ does lift along $\tau_{\leq 5} \B\Sp \to \tau_{\leq 5} \B\Sp$, and so $\Sq^2(q_1(\varrho)\,\mathrm{mod}\,2) = 0 \in H^6(\tau_{\leq 4}\B \fG, \BZ/2\BZ)$. On the other hand, since $\fG$ is connected and simply connected, $\Sq^2\colon  H^4(\tau_{\leq 4}\B \fG, \BZ/2\BZ) \to H^6(\tau_{\leq 4}\B \fG, \BZ/2\BZ)$ is injective. Thus $q_1(\varrho)\,\mathrm{mod}\,2 = 0 \in H^4(\tau_{\leq 4}\B \fG, \BZ/2\BZ) = H^4(\B \fG, \BZ/2\BZ)$, or in other words $q_1(\varrho)$ is even.
  \end{proof}

\subsection{General case (proof of Proposition~\ref{theo})}
We choose a Cartan torus $\fT\subset\fG$. The Weyl group of $(\fG,\fT)$ is denoted $W$.
If $\fG$ is simply connected, then the coweight lattice $X_*(\fT)$
coincides with the coroot lattice $Q$. The cohomology group $H^4(\B\fG,\BZ)$ is isomorphic to
the group $\on{Quad}(X_*(\fT))^W$ of $W$-invariant even-valued quadratic forms on $X_*(\fT)=Q$,
see~\S\ref{ano can}. It is in turn isomorphic to the group of $W$-invariant integer-valued
bilinear forms on $X_*(\fT)=Q$ such that $B(\lambda,\lambda)\in2\BZ$ for any $\lambda\in Q$.
Namely, $q\in\on{Quad}(X_*(\fT))^W$ goes to $B(\lambda,\mu):=\frac12(q(\lambda+\mu)-q(\lambda)-q(\mu))$.
Let $\on{Tr}\colon\fsp(2n,\BC)\times\fsp(2n,\BC)\to\BC$ stand for the trace form of the
defining representation of $\Sp(2n,\BC)$. Given a representation $\varrho\colon\fG\to\Sp(2n,\BC)$,
we obtain a bilinear form $\varrho^*\on{Tr}\in\on{Bil}(Q)^W$.
According to~Theorem~\ref{thm-tjf}, the vanishing of $\pi_4\varrho$ is equivalent to the divisibility
$\varrho^*\on{Tr}\in2\on{Bil}(Q)^W$.

Recall from~\S\ref{ano can} that for an arbitrary reductive $\fG$ with a Cartan torus $\fT$,
we denote by $\on{Bil}(X_*(\fT))^W$ the group of $W$-invariant integer-valued bilinear forms on
$X_*(\fT)$ such that $B(\lambda,\lambda)\in2\BZ$ for any $\lambda$ in the coroot sublattice
$Q\subset X_*(\fT)$. For a representation $\varrho\colon\fG\to\Sp(2n,\BC)$ we have to check the
equivalence of conditions $\pi_4\varrho=0$ and $\varrho^*\on{Tr}\in2\on{Bil}(X_*(\fT))^W$.

Note that the trace form $\on{Tr}$ on the coweight lattice of $\Sp(2n,\BC)$ assumes only
even values: $\on{Tr}(\lambda,\mu)$ is even for any coweights $\lambda,\mu$. Hence
$\varrho^*\on{Tr}$ assumes only even values for any $\varrho\colon\fG\to\Sp(2n,\BC)$
(so the problem is only to check if $\varrho^*\on{Tr}(\lambda,\lambda)$ is divisible by~4 for
any coroot $\lambda$). Now if $\fG=\fT$ is a torus, then $\pi_4(\fT)=0$, and the set of coroots
is empty. Hence the desired equivalence holds true for
any symplectic representation of any group of the form $\fG^\sico\times\fT$.

Finally, for general $\fG$, choose a finite cover $\varpi\colon\fG'\times\fT\twoheadrightarrow\fG$,
where $\fG'$ is semisimple simply-connected. It remains to check that $\varrho^*\on{Tr}$ is divisible
by~2 iff $(\varrho\circ\varpi)^*\on{Tr}$ is divisible by~2. This is clear since $\fG$ and
$\fG'\times\fT$ share the same coroots.

This completes the proof of~Proposition~\ref{theo}.

\subsection{Example}
In ``Alternative proposals (i)'' of~\cite[page 2]{t}, C.~Teleman considers an example of
homomorphism $\fG:=(\Sp(2)\times\SO(6))/\{\pm1\}\to\Sp(12)$, so that $\bM=\BC_-^2\otimes\BC_+^6$.
In this case the homomorphism $\pi_4(\Sp(2)\times\SO(6))\iso\pi_4(\fG)\to\pi_4(\Sp(12))$ vanishes.
And the pullback of the trace bilinear form from the coweight lattice of $\Sp(12)$ to the coweight
lattice of $\fG$ is divisible by~2, in accordance with~Proposition~\ref{theo}.

Indeed, the coweight lattice of $\Sp(2)\times\SO(6)$ has a natural basis
$\{\delta^*,\varepsilon_i^*,\ 1\leq i\leq 3\}$, and the coweight lattice of $\fG$ has an extra
generator $\frac12(\delta^*+\varepsilon_1^*+\varepsilon_2^*+\varepsilon_3^*)$. Note that
$\frac12(\delta^*+\varepsilon_1^*+\varepsilon_2^*+\varepsilon_3^*)$
represents a generator of $\pi_1(\fG)\simeq\BZ/4\BZ$.
The weights of the representation $\bM=\BC^2_-\otimes\BC^6_+$ are $\{\pm\delta\pm\varepsilon_i\}$.
A Lagrangian subspace has weights $\{\delta\pm\varepsilon_i\}$.
The $\Sp(12)$-weights of this Lagrangian subspace are denoted by $\{\delta_i,\ 1\leq i\leq6\}$.
Thus the homomorphism of the weight lattices takes $\delta_i,\ 1\leq i\leq 3$, to
$\delta+\varepsilon_i$, and $\delta_i,\ 4\leq i\leq 6$, to $\delta-\varepsilon_{i-3}$.
The dual homomorphism of the coweight lattices takes $\delta^*$ to $\sum_{i=1}^6\delta_i^*$,
and $\varepsilon_i^*$ to $\delta_i^*-\delta_{i+3}^*,\ 1\leq i\leq3$.
Hence $\frac12(\delta^*+\varepsilon_1^*+\varepsilon_2^*+\varepsilon_3^*)$ goes to
$\delta_1^*+\delta_2^*+\delta_3^*$. The trace pairing $\on{Tr}(\delta_i^*,\delta_j^*)=2\delta_{ij}$,
hence the trace pairing of $\delta_1^*+\delta_2^*+\delta_3^*$
with itself is equal to~6 and is divisible by~2.

Note also that the costalk of $\fR$ at the torus fixed point $\delta_1^*+\delta_2^*+\delta_3^*$
lives in cohomological degree~3. This degree is odd, so the part of
$\CA(\fG,\bM,\sqrt{\CL})=H^\bullet_{\fG_\CO}(\Gr_\fG,\CA_{\fG,\bM,\sqrt{\CL}})$ supported at the connected
component $^{(1)}\!\Gr_\fG$ of $\Gr_\fG$ containing
$\frac12(\delta^*+\varepsilon_1^*+\varepsilon_2^*+\varepsilon_3^*)$
lives in odd degrees. On the other hand, half the trace pairing of $\delta_1^*+\delta_2^*+\delta_3^*$
with itself is also odd. Hence the super line bundle $\sqrt{\CL}$ is {\em odd} at the component
$^{(1)}\!\Gr_\fG$. The total parity of $H^\bullet_{\fG_\CO}({}^{(1)}\!\Gr_\fG,\CA_{\fG,\bM,\sqrt{\CL}})$ is
{\em even}, and the resulting algebra $H^\bullet_{\fG_\CO}(\Gr_\fG,\CA_{\fG,\bM,\sqrt{\CL}})$
is commutative (as opposed to super-commutative), though it has a nontrivial odd $\BZ$-graded part.

\end{document}